\documentclass[10pt,a4paper]{article}
\usepackage[utf8]{inputenc}
\usepackage[english]{babel}
\usepackage{amsmath}
\usepackage{amsfonts}
\usepackage{amssymb}
\usepackage{amsthm}

\usepackage{color}

\usepackage{graphicx}

\usepackage[margin=2.5cm]{geometry}

\usepackage[colorlinks=false]{hyperref}

\newtheorem{theorem}{Theorem}[section]

\newtheorem{lemma}{Lemma}
\newtheorem{corollary}{Corollary}

\newtheorem{remark}{Remark}[section]

\allowdisplaybreaks

\author{Mirko D'Ovidio}

\title{Fractional Boundary Value Problems and Elastic Sticky Brownian Motions, I: The half line}

\begin{document}
\maketitle

\begin{abstract}
We extend the results obtained in \cite{Dov22} by introducing a new class of boundary value problems involving non-local dynamic boundary conditions. We focus on the problem to find a solution to a local problem on a domain $\Omega$ with non-local dynamic conditions on the boundary $\partial \Omega$. Due to the pioneering nature of the present research, we propose here the apparently simple case of $\Omega=(0, \infty)$ with boundary $\{0\}$ of zero Lebesgue measure. Our results turn out to be instructive for the general case of boundary with positive (finite) Borel measures. Moreover, in our view, we bring new light to dynamic boundary value problems and the probabilistic description of the associated models.  
\end{abstract}

\maketitle

\tableofcontents

\vspace{1cm}

\\

{\bf MCS2020 subject classifications:} 60J50 ; 60J55 ; 35C05 ; 26A33

\section{Introduction}

\subsection{Presentation of the results}

We study the solution to the (local) heat equation subjected to a non-local dynamic boundary condition. The non-local condition is written in terms of the Caputo-Dzherbashian fractional derivative defined as the convolution operator
\begin{align*}
D^\alpha_t \varphi(t) = (\varphi^\prime * \kappa)(t), \quad \alpha \in (0,1)
\end{align*}
involving the singular kernel $\kappa$ described below (see Section \ref{sec:CD-der}). Our results can be generalized by considering a general kernel $\kappa$ and therefore a non-local operator (see $D^\Phi_t$ below) as defined in \cite{Koc2011,Toaldo2015,Chen17,Kolo19}. Here we maintain our focus on $D^\alpha_t$ and refer to \cite{BonColDov} for general non-local boundary conditions. The Caputo-Dzherbashian derivative has been introduced in \cite{caputoBook,CapMai71,CapMai71b} by the first author and separately in a series of works starting from \cite{Dzh66,DzhNers68} by the second author.

The fractional Boundary condition introduces an anomalous behaviour of the associated process only on the lower dimensional space given by the boundary. In the present paper we extend the result given in \cite{Dov22} and provide a discussion in case a Brownian motion on $(0, \infty)$ has an anomalous behaviour only at the boundary point $\{0\}$. An interesting formulation of the problem would be given for an arbitrary domain $\overline{\Omega} = \Omega \sqcup \partial \Omega$ as the problem to find a solution $u$ such that
\begin{equation}
\label{problemINTRO}
\left\lbrace
\begin{array}{l}
\displaystyle \frac{\partial u}{\partial t} = A u \quad \textrm{on } \Omega,\\
\displaystyle \mathfrak{u} = Tu,\\
\displaystyle  D^\alpha_t \mathfrak{u} = Bu  \quad \textrm{on } \partial \Omega,\\
\displaystyle  u_0 = f
\end{array}
\right.
\end{equation}
where the trace operator $T$ must be well-defined on $L^2(m_\partial)$ for a suitable finite Borel measure $m_\partial$ supported on $\partial \Omega$. Thus, we may focus on the problem to find a continuous kernel for $u$ on $L^2(\overline{\Omega}, m)$ where $m$ is obtained as the sum of the Lebesgue measure $dx$ and the surface measure $m_\partial$, that is $L^2(\overline{\Omega}, m)= L^2(\Omega, dx) \oplus L^2(\partial \Omega, m_\partial)$. Notice that the non-local boundary condition says that we can not obtain a semigroup for the problem \eqref{problemINTRO}. Although the problem \eqref{problemINTRO} is interesting, there are no results on this problem and we are lead to approach the problem by first considering the instructive case $\overline{\Omega}=[0, \infty)$ with $\partial \Omega = \{0\}$. For a smooth domain $\Omega \subset \mathbb{R}^d$ with $d>1$ we refer to the work \cite{FBVParXiv}. For $\Omega=[0, \infty)$ we refer to \cite{Dov22} where the fractional dynamic boundary condition has been introduced for a problem with constant initial datum. In order to prepare our discussion we first fix below some notation and ideas. \\

As usual $\mathbf{E}_x$ is the expectation with respect to a probability measure $\mathbf{P}_x$ with $x$ denoting the starting point for the process we are dealing with. We also use the following notation:
\begin{align*}
\dot{\varphi} = \frac{\partial \varphi}{\partial t}, \quad \varphi^\prime = \frac{\partial \varphi}{\partial x}, \quad \varphi^{\prime \prime} = \frac{\partial^2 \varphi}{\partial x^2}.
\end{align*}

Let $\eta, \sigma, c$ be positive constants. We consider the following processes:
\begin{itemize}
\item[i)] $X^+= \{X^+_t\}_{t\geq 0}$ is a reflecting Brownian motion on $[0, \infty)$ with boundary local time $\gamma^+ = \{\gamma^+_t\}_{t\geq 0}$;
\item[ii)] $X^\dagger = \{X^\dagger_t\}_{t\geq 0}$ is a Brownian motion on $(0, \infty)$ killed upon reaching the boundary $\{0\}$ for which
\begin{align*}
\mathbf{E}_x[f(X^\dagger_t)] = \mathbf{E}_x[f(X^+_t), t< \tau_0], \quad t>0,\; x \in (0,\infty)
\end{align*}
where $\tau_0 = \inf \{t \geq 0\,:\, X^+_t =0 \}$;
\item[iii)] $X= \{X_t\}_{t\geq 0}$ is an elastic sticky Brownian motion on $[0, \infty)$ with boundary local time $\gamma=\{\gamma_t\}_{t \geq 0}$. The term \lq\lq elastic sticky \rq\rq is referred to the fact that, for $c\geq 0$,
\begin{align*}
\mathbf{E}_x[f(X^{el}_t)] = \mathbf{E}_x[f(X^+_t) M_t], \quad M_t := \exp(-c \gamma^+_t)
\end{align*}
is the well-known representation of the elastic Brownian motion $\{X^{el}_t\}_{t\geq 0}$ in terms of $X^+$ and the multiplicative functional $M^c_t$. Passing through the random time change $V_t = t + (\eta/\sigma) \gamma^+_t$ we therefore use the representation (a well-known result given in \cite{ItoMcK})
\begin{align*}
\mathbf{E}_x[f(X_t)] = \mathbf{E}_x[f(X^{el} \circ V^{-1}_t)]
\end{align*} 
where $V^{-1}_t := \inf\{s\,:\, V_s >t\}$ is the inverse of $V_t$. We recall the associated boundary condition
\begin{align}
\label{BCintro}
\eta \varphi^{\prime \prime} = \sigma \varphi^\prime - c \varphi
\end{align}
which is termed Wentzell-Robin or Feller-Robin boundary condition and includes a (pure) sticky condition obtained as $\eta \to \infty$ and an elastic condition obtained as $\eta \to 0$. We say that $X$ has generator $(G, D(G))$ where $G\varphi = \varphi^{\prime \prime}$ and 
\begin{align*}
D(G) = \{\varphi \in C^2([0, \infty))\,:\, \varphi \textrm{ satisfies \eqref{BCintro} at the boundary point $\{0\}$}\}.
\end{align*}
Let us write $\gamma^+_{t}=\gamma^+_{t,0}$ where $\gamma^+_{t,z}$ is the jointly continuous local time of $X^+$ for the level $z \in [0, \infty)$ and time $t\geq 0$ for which, given a bounded measurable function $f:[0, \infty) \mapsto [0, \infty)$, the following occupation time formula holds
\begin{align*}
\int_0^t f(X^+_s)ds = \int_0^\infty f(z)\gamma^+_{t,z}dz, \quad t\geq 0.
\end{align*}
The boundary point $\{0\}$ must be carefully treated. We therefore consider the measure  
\begin{align*}
m(dz) = dz + (\eta/\sigma) \delta_0(dz)
\end{align*}
on $\overline{\Omega} = [0, \infty)$ given by the sum of the Lebesgue measure on $(0, \infty)$ and the Dirac measure at $\{0\}$. We observe that
\begin{align}
\label{Vbym}
\int_{[0, \infty)} \gamma^+_{t,z} m(dz) = \int_0^\infty \gamma^+_{t,z} dz + (\eta/\sigma) \gamma^+_{t,0} = t + (\eta/\sigma) \gamma^+_{t}=:V_t
\end{align}
gives the random time introduced above;
\item[iv)] $H=\{H_t\}_{t\geq 0}$ is a stable subordinator of order $\alpha \in (0,1)$;
\item[v)] $L=\{L_t\}_{t\geq 0}$ is the inverse $L_t := \inf\{s \geq 0\,:\, H_s >t\}$ to $H$;
\item[vi)] $\bar{X} = \{\bar{X}_t\}_{t \geq 0}$ with boundary local time $\bar{\gamma}= \{\bar{\gamma}_t\}_{t\geq 0}$ introduced and discussed below in the present section.
\end{itemize}
We are now ready to discuss our results. \\

Let $\eta, \sigma, c$ be positive constants. We consider the following problems:
\begin{align}
\label{P1-FIVP}
\tag{P1}
\left\lbrace
\begin{array}{ll}
\displaystyle D^\alpha_t v(t,x) = v^{\prime \prime}(t,x), \quad t>0, \, x >0\\
\\
\displaystyle \eta v^{\prime \prime}(t,0) = \sigma v^{\prime}(t,0) - c\, v(t, 0), \quad t>0\\
\\
\displaystyle v(0,x) = f(x), \quad x \geq 0
\end{array}
\right.
\end{align}
\begin{align}
\label{P2-FP}
\tag{P2}
\left\lbrace
\begin{array}{ll}
\displaystyle D^\alpha_t w(t,x) = w^{\prime \prime}(t,x), \quad t>0, \, x >0\\
\\
\displaystyle \eta D^\alpha_t w(t,0) = \sigma w^\prime(t,0) - c\, w(t, 0), \quad t>0\\
\\
\displaystyle w(0,x) = f(x), \quad x \geq 0
\end{array}
\right.
\end{align}
\begin{align}
\label{P3-FBVP}
\tag{P3}
\left\lbrace
\begin{array}{ll}
\displaystyle \dot{u}(t,x) = u^{\prime \prime}(t,x), \quad t>0, \, x >0\\
\\
\displaystyle \eta D^\alpha_t u(t,0) = \sigma u^\prime(t,0) - c\, u(t, 0), \quad t>0\\
\\
\displaystyle u(0,x) = f(x), \quad x \geq 0
\end{array}
\right.
\end{align}

The problem \eqref{P1-FIVP} is a fractional initial value problem (FIVP) also called fractional Cauchy problem (FCP). It has been investigated by many researchers, see for example \cite{Koc89,Kiryakova94,OB09,MNV09,Baz2000,BaeMee2001,Dov12,GLY15,DovNan16,MLP2001,luchko20}. All these works basically consider the semigroup associated with $\alpha=1$ as a base object to deal with. Then, for $\alpha \in (0,1)$, the solution to the FCP comes out, via integration of that semigroup with a suitable kernel. From the probabilistic point of view we have a time change for a Markov (base) process where the new random time is given by $L$. The composition is not Markov and the associated operator is not a semigroup. 

Our results have some impact on occupation measures, then we think the following discussion will help the readers. Let $Y=\{Y_t\}_{t\geq 0}$ be a Markov process with generator $(A_Y, D(A_Y))$. Then, the solution to the FCP 
\begin{align}
D^\alpha_t \varphi = A_Y \varphi, \quad \varphi_0= f \in D(A_Y)
\label{FCPexample}
\end{align}
has the probabilistic representation 
\begin{align*}
\varphi(t,x) = \mathbf{E}_x[f(Y^L_t)] = \int_0^\infty \mathbf{E}_x[f(Y_s)]\, \mathbf{P}_0(L_t \in ds)
\end{align*}
where $Y^L = \{Y^L_t\}_{t\geq 0}$ defined as $Y^L_t = Y \circ L_t$ is obtained via time change from the base process $Y$. For $Y$ on a bounded domain $\Omega$ and $\forall\, \Lambda \subset \Omega$ we introduce $\tau_Y(\Lambda) = \inf\{t\,:\, Y_t \notin \Lambda \}$ and $\tau^L_Y(\Lambda) = \inf \{t\,:\, Y^L_t \notin \Lambda\}$ as the first exit time from the subset $\Lambda$ respectively of $Y$ and $Y^L$. For $x \in \Lambda$, $S \subset \Lambda$, we have (\cite{Chen17})
\begin{align*}
\mu_Y(x, S) := \mathbf{E}_x \left[ \int_0^{\tau_Y(\Lambda)} \mathbf{1}_S(Y_s) ds \right] < \infty
\end{align*}
and
\begin{align*}
\mu_Y^L(x, S) := \mathbf{E}_x \left[ \int_0^{\tau^L_Y(\Lambda)} \mathbf{1}_S(Y^L_s) ds \right] = \infty,
\end{align*}
for the occupation measures $\mu_Y$ and $\mu_Y^L$ of a subset $S$. This result well agrees with the fact that (\cite{CapDovDelRus17,Chen17})
\begin{align*}
\mathbf{E}_x[\zeta] < \infty, \quad \mathbf{E}_x[\zeta^L] = \infty 
\end{align*}
where $\zeta$ and $\zeta^L=H\circ \zeta$ respectively are the lifetimes of $Y$ and $Y^L$ on $\Omega$ if we allow a kill on $\partial \Omega$ (Dirichlet condition on the boundary). Thus, we have infinite average amount of time spent by the process $Y^L_t$ in $\Omega$ and every subsets $\Lambda$ of $\Omega$. We say that the process is delayed by $L$. 

\begin{remark}
This can be better explained by considering the general operator
\begin{align*}
D^\Phi_t \varphi = \varphi^\prime * \kappa, \quad \kappa(z) = \Pi(z, \infty)  
\end{align*}
where
\begin{align}
\label{genSymbH}
\Phi(\lambda) =\int_0^\infty (1-e^{-\lambda z}) \Pi(dz), \quad \textrm{s.t.} \quad \int_0^\infty (1\wedge z) \Pi(dz) <\infty
\end{align}
is the symbol of a subordinator with L\'evy measure $\Pi$ and $\kappa$ is the tail of $\Pi$ with 
\begin{align*}
\int_0^\infty e^{-\lambda z} \kappa(z)dz = \frac{\Phi(\lambda)}{\lambda}. 
\end{align*}
The fact that $E_x[\zeta^L] < \infty$ can be related with the elliptic problem associated to the FCP and the integrability of $D^\Phi_t$. In particular, from the problem \eqref{FCPexample}, via integration with respect to time, we get 
\begin{align*}
\int_0^\infty D^\alpha_t \varphi = A_Y \int_0^\infty \varphi dt, \quad \varphi_0 = 1
\end{align*}
and
\begin{align*}
\int_0^\infty \varphi dt = \int_0^\infty \mathbf{P}_x(\zeta > t) dt
\end{align*}
gives the mean lifetime of $Y^L$. Notice that $\zeta$ coincides with a killing time if we assume Dirichlet boundary condition. The elliptic problem associated with \eqref{FCPexample} makes sense according with 
\begin{align*}
\|D^\Phi_t \varphi \|^p_p \leq \| \varphi \|^p_p \left( \lim_{\lambda \to \infty} \frac{\Phi(\lambda)}{\lambda} \right)^p, \quad p \in [1, \infty)
\end{align*}
for $\varphi \in L^p(0, \infty)$ and $\lim_{\lambda \to \infty} \Phi(\lambda)/\lambda < \infty$. The derivative $D^\alpha_t$ is the special case with $\Phi(\lambda)=\lambda^\alpha$ for which the limit above is infinite. In \cite{CapDovDelRus17} we discuss this property as a delaying effect for the base process $Y$.
\end{remark}

The problem \eqref{P2-FP} can be associated with \eqref{P1-FIVP}. Indeed, it can be regarded as a simple extension to the well-known results about fractional PDE in case the following assumption turns out to be verified:
\begin{align}
\label{AssumptionA1}
\tag{A1}
\forall\, t>0 \quad \lim_{x\to 0} \big( D^\alpha_t w(t, x) - w^{\prime \prime}(t,x) \big)= 0. 
\end{align}
Thus, under \eqref{AssumptionA1} the identities 
\begin{align*}
\eta D^\alpha_t w(t, 0) \stackrel{(A1)}{=} \eta w^{\prime \prime}(t,0) \stackrel{\eqref{BCintro}}{=} \sigma w^\prime (t,0) - c\, w(t, 0), \quad t>0
\end{align*}
give the equivalence between \eqref{P1-FIVP} and \eqref{P2-FP}. We observe that \eqref{AssumptionA1} can be considered also in a weaker sense. However, the key role here is played by the continuity of the second derivative at zero. We observe that an equivalent formulation of these problems can be given as
\begin{align*}
D^\alpha_t \left( \begin{matrix} \varphi\\ \varphi|_0 \end{matrix} \right)
= \mathcal{A} \left( \begin{matrix} \varphi\\ \varphi|_0 \end{matrix} \right)
\end{align*}
with the suitable definition of the operator matrix $\mathcal{A}$ on $C(\overline{\Omega})$. This does not hold for the next case.\\

The problem \eqref{P3-FBVP} involves a non-local equation on the boundary which can be written as 
\begin{align*}
\eta D^\alpha_t u(t,0) 
= & \lim_{x\downarrow 0} \left( \eta \frac{1}{\Gamma(1-\alpha)} \int_0^\infty u^{\prime \prime}(s, x) (t-s)^{-\alpha} ds \right)\\
= & \eta \frac{1}{\Gamma(1-\alpha)} \int_0^\infty u^{\prime \prime}(s, 0) (t-s)^{-\alpha} ds, \quad t>0.
\end{align*}
by exploiting the definition of $D^\alpha_t$. Our main goal in the present work is to study this problem and obtain a probabilistic description of the associated models. In some sense we extend the results obtained in \cite{ItoMcK} and \cite{Wentzell59} on the second order boundary condition. In probability, we refer to the sticky condition. A bridge between second order boundary conditions and dynamic boundary conditions can be given for the heat equation under the fact that $\dot{u}(t, x) |_{x=0} = \dot{u}(t, 0)$ as in the Assumption \eqref{AssumptionA1}. Thus, a Feller-Wentzell boundary condition can be treated as a dynamic boundary condition for the heat equation.  

The probabilistic representation of the problem \eqref{P1-FIVP} is concerned with the process $\bar{X}$. We show (Theorem \ref{thm:FBVP-realLine}) that $\bar{X}$ admits the representation $\bar{X}_t = X^{el} \circ \bar{V}^{-1}_t$ where $\bar{V}^{-1}_t$ is the inverse to the process
\begin{align}
\bar{V}_t = t + H \circ (\eta/\sigma) \gamma^+_t, \quad t\geq 0.
\end{align}
Notice that, for $\alpha=1$ $H_t=t$ a.s. and therefore $\bar{V}_t = V_t$ a.s., thus $\bar{X}$ coincides with $X$ for $\alpha=1$. Let us define $\tau_{\bar{X}}(\Lambda) = \inf\{t\,:\, \bar{X}_t \notin \Lambda \}$ and $\tau_X(\Lambda) = \inf\{t\,:\, X_t \notin \Lambda\}$. For $\alpha \in (0,1]$, we have finite occupation measures, that is, for $S \subset \Lambda$, 
\begin{align*}
\forall x \in \Lambda \subset (0, \infty), \quad \mu_{\bar{X}}(x, S) := \mathbf{E}_x\left[ \int_0^{\tau_{\bar{X}}(\Lambda)} \mathbf{1}_S (\bar{X}_s)ds \right] < \infty.
\end{align*}
This is related with the fact that $\bar{X}$ (and $X$ for $\alpha=1$) behaves like a Brownian motion on $(0, \infty)$. It moves on finite intervals with finite times. For the time the process $\bar{X}$ spends on the boundary point we refer, as usual, to the holding times. Here we have different behaviours depending on $\alpha$. It is well-known that $X$ spends an exponential time on the boundary with each visit, we denote by $\{e_i\}_i$ the sequence of i.i.d. holding times for $X$. The average amount of time the process $X$ spends at $\{0\}$ is therefore given by
\begin{align*}
\mathbf{E}[e_i] = (\eta/\sigma), \quad \forall\, i.
\end{align*}
Concerning the process $\bar{X}$ we refer to the sequence $\{\bar{e}_i\}_{i}$ of i.i.d. holding times, that is $\bar{e}_i$ for the visit $i$ is a random time such that, for $\bar{X}_0=0$ and $\bar{X}_{\bar{e}_i}>0$, it holds (Theorem \ref{thm:holdingTime})
\begin{align*}
\mathbf{P}_0(\bar{e}_i > t | \bar{X}_{\bar{e}_i} > 0) = E_\alpha(-(\sigma/\eta) t^\alpha), \quad t\geq 0
\end{align*}
where $E_\alpha$ is the Mittag-Leffer function. It is well-known that $E_\alpha \notin L^1(0,\infty)$. The time the process $\bar{X}$ spends on the boundary point $\{0\}$ with each visit, is a Mittag-Leffler random variable with
\begin{align*}
\forall\, \alpha \in (0,1), \quad \mathbf{E}[\bar{e}_i ] = \infty, \quad \forall\, i.
\end{align*}
This special behaviour of $\bar{X}$ for $\alpha \in (0,1)$ entails a new role for the boundary point $\{0\}$. This point can be regarded as a trap for the Brownian motion on $(0, \infty)$.

\subsection{Plan of the work}

We begin with some preliminaries. In Section \ref{sec:FIVP} we compare the problems \eqref{P1-FIVP} and \eqref{P2-FP}. In Section \ref{sec:FBVP} we present our main results on the problem \eqref{P3-FBVP}.

\subsection{Motivations}
The present work brings new light to the boundary value problems with dynamic conditions. Although the main interest would be apparently given by motions on bounded domains $\Omega \subset \mathbb{R}^d$ with $d>1$ our problem could be very attractive for a number of applications. Indeed, the study of non-local dynamic problems on higher dimension introduces motions on the lower dimensional space $\partial \Omega \subset \mathbb{R}^{d-1}$. This is obviously of interest and reasonably leads to pure jump processes on $\partial \Omega$, the associated trace process. This is motivated by our construction given in terms of $X^+$, instantaneous reflections may only give holding times (we do not have boundary motions). On the other hand, the fact that we deal with a zero Lebesgue measure boundary, the point $\{0\}$, it does not seem to be restrictive. Indeed, the independent holding times at $\{0\}$ (or a general point $x>0$) for a Brownian motion find interesting applications in different fields of applied sciences. We mention the financial models in which investors update their beliefs too slowly or traffic models in which we construct motions on metric graphs with (independent) delay on a vertex. In general, the sticky point for the process $X$ becomes a trap point for the process $\bar{X}$. In our analysis, the boundary point $\{0\}$ can be regarded as a trap point for the Brownian motion $X^+$ on $[0, \infty)$.

\section{Preliminaries}

\subsection{The Caputo-Dzherbashian fractional derivative}
\label{sec:CD-der}

For a function $\varphi(\cdot,x) : (0, \infty) \mapsto (0, \infty)$, $\forall\, x \in [0, \infty)$ we consider the convolution-type operator defined as
\begin{align*}
D^\alpha_t \varphi(t,x) = \frac{1}{\Gamma(1-\alpha)} \int_0^t \frac{\partial \varphi}{\partial s}(s,x)\, (t-s)^{-\alpha} ds
\end{align*}
for $\alpha \in (0,1)$. \\

We notice that the Caputo-Dzherbashian derivative is well-defined for functions which are exponentially bounded together with their first derivative. In particular, $D^\alpha_t \varphi$ is well defined for $\varphi(\cdot, x) \in W^{1,\infty} (0,\infty)$, $\forall\, x \in [0, \infty)$ and this ensures existence of the Laplace transform
\begin{align}
\int_0^\infty e^{-\lambda t} D^\alpha_t \varphi(t,x)\, dt 
= & \left( \int_0^\infty e^{-\lambda t} \frac{t^{-\alpha}}{\Gamma(1-\alpha)} \right) \left(\int_0^\infty e^{-\lambda t} \frac{\partial \varphi}{\partial t}(t,x) \, dt \right)\notag \\
= & \frac{\lambda^\alpha}{\lambda} \big(\lambda \tilde{\varphi}(\lambda, x) - u(0, x) \big), \quad \lambda>0
\label{CD-LT}
\end{align}
where
\begin{align*}
\tilde{\varphi}(\lambda, x) = \int_0^\infty e^{-\lambda t} \varphi(t,x) dt, \quad \lambda>0.
\end{align*}
Formula \eqref{CD-LT} can be obtained by observing that $D^\alpha_t$ is a convolution operator. Notice that we do not require $\varphi(\cdot, x) \in L^1(0, \infty)$ for some $x$.

\subsection{The process $X$}
\label{sec:secX}

Our discussion is mainly based on the well-known book \cite{BluGet68} and the pioneering work \cite{ItoMcK}. Let us consider the natural filtration $ \mathcal{F}_t = \sigma\{X_s, \, 0\leq s < t\}$ and a good function $f$ for which $\mathbf{E}[f(X_s)| \mathcal{F}_t] = \mathbf{E}[f(X_s)|X_t]$, $t\leq s$, and $\mathbf{E}_x [f(X_{t+s}) | \mathcal{F}_t] = \mathbf{E}_{X_t} [f(X_s)]$, $s,t>0$. We say that $X$ is an elastic process meaning that $\mathbf{E}_x[f(X_t)] = \mathbf{E}_x[f(X_t), t < \zeta]$ is written in terms of the multiplicative functional $M_t= \mathbf{1}_{(t< \zeta)}$ where the lifetime $\zeta$ well accords with an elastic kill. There exists an independent exponential random variable (with parameter $c/\eta$) for which
\begin{align}
\label{lawExpLT}
\mathbf{E}[M_t | X_t] = e^{-(c/\eta) \gamma_t}.
\end{align}
On the other hand, from the sticky condition, $\{t\,:\, X_t \in \partial \Omega\}$ is a Lebesgue measurable set obtained from the holding times of $X$ on $\partial \Omega$. In particular, we may consider a sequence $\{e_i\}_i$ of (identically distributed) independent exponential random variables (with parameter $\sigma/\eta$) for which $\mathbf{P}_x(e_1 > t ,  X_{e_1} \in dy) = e^{-(\sigma/\eta)t} \mathbf{P}_x(X_{e_1} \in dy)$ for $x \in \partial \Omega$ and
\begin{align}
\label{lawHoldingTimeX}
\mathbf{P}_x(e_1 > t | X_{e_1}) = e^{-(\sigma/\eta) t}, \quad x \in \partial \Omega.
\end{align}
As announced we consider $\overline{\Omega} = [0, \infty)$ with the boundary point $\{0\}$. Since $X$ is an elastic sticky Brownian motion we have the representation (\cite[Section 10]{ItoMcK} 
\begin{align}
\label{repXelSticky}
\mathbf{E}_x[f(X_t)] = \mathbf{E}_x\left[ f(X^+ \circ V^{-1}_t) \exp \left( -c/\sigma\, \gamma^+ \circ V^{-1}_t \right) \right]
\end{align}
where 
\begin{align*}
V^{-1}_t = \inf\{s\geq 0\,:\, V_s >t\}
\end{align*}
is the inverse of
\begin{align*}
V_t = t + (\eta/\sigma) \gamma^+_t.
\end{align*}
Notice that $V_t$ and $V^{-1}_t$ are both continuous and strictly increasing. Moreover, 
\begin{align*}
V \circ V^{-1}_t = t, \quad t\geq 0.
\end{align*} Under the representation \eqref{repXelSticky} for $X$, the boundary local time $\gamma$ of $X$ can be given as the composition (see \cite{ItoMcK}, formula 18, Section 10)
\begin{align}
\label{equivgamma}
\gamma_t = (\eta/\sigma)\gamma^+ \circ V^{-1}_t, \quad t>0.
\end{align}
We recall the resolvent
\begin{align}
\label{potXdef}
R_\lambda f(x) = e^{-x \sqrt{\lambda}} \frac{\sigma \int_0^\infty e^{-y \sqrt{\lambda}} f(y) dy + \eta f(0)}{c + \eta \lambda + \sigma \sqrt{\lambda}} + R^\dagger_\lambda f(x), \quad x \in [0, \infty),\; \lambda>0
\end{align}
where $R^\dagger_\lambda f = \int_0^\infty e^{-\lambda t} Q^\dagger_t f dt$ is the resolvent for the Dirichlet semigroup $Q^\dagger_t$. For the explicit calculation we refer to \cite{ItoMcK} and \cite{Dov22}. Moreover, we write
\begin{align*}
\mathbf{P}_x(X_t \in dy) = p(t,x,y) dy
\end{align*}
where $p$ is the continuous kernel of $X$.

\subsection{The random times $L$ and $H$}
The process $L$ is the inverse to the $\alpha$-stable subordinator $H$ defined by $L_t = \inf\{s\geq 0\,:\, H_s>t \}$  and for which $\mathbf{P}_0(L_t < s) = \mathbf{P}_0(t < H_s)$, $t,s>0$. We assume that $H_0=0=L_0$ and write $\mathbf{P}_0$ for the associated probability measure. Denote by $l$ and $h$ the corresponding probability densities for which
\begin{align*}
\mathbf{P}_0(H_t \in ds) = h(t,s)ds, \quad \mathbf{P}_0(L_t \in ds) = l(t,s)ds.
\end{align*}
Then
\begin{align}
\label{LapHandL}
\int_0^\infty e^{-\xi s} h(t,s) ds = e^{-t \xi^\alpha}, \quad \int_0^\infty e^{-\lambda t} l(t,s)dt = \frac{\lambda^\alpha}{\lambda} e^{-s \lambda^\alpha}, \quad \xi, \lambda>0.
\end{align}
We recall that
\begin{align}
\label{LapML}
\int_0^\infty e^{-\xi s} l(t,s) ds = E_\alpha(- \xi t^\alpha) \quad \textrm{with} \quad \int_0^\infty e^{-\lambda t} E_\alpha (-\xi t^\alpha) dt = \frac{\lambda^{\alpha-1}}{\lambda^\alpha + \xi}, \quad \lambda>0
\end{align}
where the Mittag-Leffler function $E_\alpha$ is analytic and such that
\begin{align}
\label{MLbound}
| E_\alpha(-\mu t^\alpha)| \leq \frac{C}{1+ \mu t^\alpha}, \quad t\geq 0, \; \mu>0 \quad \textrm{for a constant } C>0 
\end{align}
(see \cite{Bingham71,Kra03}). We underline that $E_\alpha \notin L^1(0, \infty)$ for $\alpha \in (0,1)$.

We recall that $\lambda^\alpha$ is the so-called symbol of $H$ such that $\mathbf{E}_0[\exp -\lambda H_t] = \exp -t \lambda^\alpha$, $\lambda>0$. For the reader's convenience we also recall that $\lambda^\alpha$ is a Bernstein function with representation (according with \eqref{genSymbH})
\begin{align*}
\lambda^\alpha = \int_0^\infty \left( 1 - e^{-\lambda z} \right) \frac{\alpha}{\Gamma(1-\alpha)} z^{-\alpha-1} dz, \quad \lambda \geq 0.
\end{align*}
Moreover,
\begin{align*}
\overline{\Pi}(z) := \int_z^\infty \frac{\alpha}{\Gamma(1-\alpha)} y^{-\alpha-1} dy = l(z,0), \quad z >0
\end{align*}
as the Laplace transform of both sides entails.

\section{The fractional initial value problem}
\label{sec:FIVP}

\subsection{Time changes for FIVPs}

\begin{theorem}
Let us consider the solution $v$ to the problem \eqref{P1-FIVP} and the solution $w$ to the problem \eqref{P2-FP}:
\begin{itemize}
\item[i)] $v,w \in D(G)$;
\item[ii)] $\forall\, t\geq 0,\; \forall\, x \in [0, \infty),\; v(t,x) = w(t,x)$.
\end{itemize}
Moreover,
\begin{align*}
\int_0^\infty e^{-\lambda t} v(t,x) dt = \frac{\lambda^\alpha}{\lambda} R_{\lambda^\alpha} f(x), \quad \lambda >0,\; x \in [0, \infty)
\end{align*}
where $R_\lambda$ has been defined in \eqref{potXdef} and the following probabilistic representation holds true
\begin{align}
v(t,x) = \mathbf{E}_x[f(X \circ L_t)], \quad t \geq 0,\; x \in [0, \infty).
\end{align}
\end{theorem}
\begin{proof}
Let us consider the problem \eqref{P1-FIVP}. The problem 
\begin{align}
\label{FIVP-realLine}
\left\lbrace
\begin{array}{ll}
\displaystyle D^\alpha_t \varphi(t,x) = \varphi^{\prime \prime}(t,x), \quad t>0, \, x >0\\
\\
\displaystyle \eta \varphi^{\prime \prime}(t,0) = \sigma \varphi^\prime(t,0) - c\, \varphi(t, 0), \quad t>0\\
\\
\displaystyle \varphi(0,x) = f(x), \quad x \geq 0, \quad f \in C_b[0, \infty)
\end{array}
\right.
\end{align}
has a unique solution with probabilistic representation
\begin{align}
\label{repP1P2proof}
\varphi(t,x) = \mathbf{E}_x[f(X \circ L_t)] = \mathbf{E}_x[f(X^+ \circ V^{-1} \circ L_t) \, M \circ V^{-1} \circ L_t ].
\end{align}
To prove our statement we first assume that the representation \eqref{repP1P2proof} holds true. Then we consider the $\lambda$-potential
\begin{align*}
\int_0^\infty e^{-\lambda t}  \varphi(t, x) dt 
= & \frac{\lambda^\alpha}{\lambda} \mathbf{E}_x \left[ \int_0^\infty e^{-\lambda^\alpha t} f(X^+ \circ V^{-1}_t) M \circ V^{-1}_t \, dt \right] = \frac{\lambda^\alpha}{\lambda} R_{\lambda^\alpha} f(x)
\end{align*}
where $R_\lambda f$ has been introduced in \eqref{potXdef}. Thus
\begin{align*}
\lim_{x\to 0} \int_0^\infty e^{-\lambda t}  \varphi(t, x) dt = \frac{\lambda^\alpha}{\lambda} \frac{\sigma \int_0^\infty e^{-y \sqrt{\lambda^\alpha}} f(y) dy + \eta f(0)}{c + \eta \lambda^\alpha + \sigma \sqrt{\lambda^\alpha}} =: \widetilde{\varphi}(\lambda, 0)
\end{align*}
We get
\begin{align*}
\eta \lambda^\alpha \widetilde{\varphi}(\lambda, 0) - \frac{\lambda^\alpha}{\lambda} \eta f(0) = \sigma \frac{\lambda^\alpha}{\lambda} \int_0^\infty e^{-y \sqrt{\lambda^\alpha}} f(y) dy - \sigma \sqrt{\lambda^\alpha} \widetilde{\varphi}(\lambda, 0) - c \widetilde{\varphi}(\lambda, 0) 
\end{align*}
where
\begin{align*}
\eta \lambda^\alpha \widetilde{\varphi}(\lambda, 0) - \frac{\lambda^\alpha}{\lambda} \eta f(0) = \int_0^\infty e^{-\lambda t} \eta D^\alpha_t \varphi(t,0) dt
\end{align*}
and
\begin{align*}
\sigma \frac{\lambda^\alpha}{\lambda} \int_0^\infty e^{-y \sqrt{\lambda^\alpha}} f(y) dy - \sigma \sqrt{\lambda^\alpha} \widetilde{\varphi}(\lambda, 0) - c \widetilde{\varphi}(\lambda, 0)
\end{align*}
equals
\begin{align*}
\int_0^\infty e^{-\lambda t} \left(  \sigma \frac{\partial \varphi}{\partial x}(t,0) - c\, \varphi(t, 0) \right) dt.
\end{align*}
We conclude that $\varphi = w$, the solution of \eqref{P2-FP}.\\

On the other hand, 
\begin{align}
\lim_{x\to 0} \int_0^\infty e^{-\lambda t} \varphi^{\prime \prime}(t,x) dt 
= & \lim_{x\to 0} \left( \frac{\lambda^\alpha}{\lambda} R_{\lambda^\alpha} f(x) \right)^{\prime \prime}  \\
= & \frac{\lambda^\alpha}{\lambda} \left( \lambda^\alpha R_{\lambda^\alpha} f(0) - f(0) \right) \notag\\
= & \int_0^\infty e^{-\lambda t} \dot{\varphi}(t,0)\, dt. \notag
\end{align}
Notice that
\begin{align*}
\frac{\lambda^\alpha}{\lambda} \left( \lambda \frac{\lambda^\alpha}{\lambda} R_{\lambda^\alpha } f(0) - f(0) \right) = \left( \int_0^\infty e^{-\lambda t} \overline{\Pi}(t)\,dt \right) \left( \int_0^\infty e^{-\lambda t} \varphi^\prime(t, 0) dt \right)
\end{align*}
and therefore
\begin{align*}
\lim_{x\to 0} \int_0^\infty e^{-\lambda t} \varphi^{\prime \prime}(t,x) dt  = ( \varphi^\prime * \overline{\Pi})(t, 0)
\end{align*}
which is the fractional derivative $D^\alpha_t \varphi$ at $x=0$. We conclude that $\varphi = v$ is the solution of \eqref{P1-FIVP}. 

In both cases we have a continuous inverse of the Laplace transform, then there exist at most one inverse. The uniqueness of the continuous solutions for the two problems says that $w=v$ pointwise on $[0, \infty) \times [0, \infty)$. 

\end{proof}

\section{The fractional boundary value problem}
\label{sec:FBVP}

We recall that by FBVP we mean a local equation (the heat equation) on $(0, \infty)$ equipped with non-local condition at the boundary point $\{0\}$. In the literature many authors use to refer to a FBVP in case of non-local equations with local boundary conditions, that is, in our paper, a FIVP. 

\subsection{The process $\bar{X}$ and the FBVP}

We state here the main result of the work.\\

Further on we denote by $\bar{p}=\bar{p}(t,x,y)$ and $p^+=p^+(t,x,y)$ the continuous kernels associated with $\bar{X}$ and $X^+$. It will be clear that $\bar{X}$ becomes a Markov process only in case $\alpha=1$, that is the case $\bar{X}=X$. For $\alpha \in (0,1)$, $\bar{p}$ can not be associated with a semigroup operator. We also underline that $X$ does not enjoy the strong Markov property because of the holding times at $\{0\}$. However, the holding times for $X$ are exponential random variables and this maintains the Markov property.\\

We recall that $X^+$ is a reflected Brownian motion on $[0, \infty)$ with local time
\begin{align*}
\gamma^+_t = \int_0^t \mathbf{1}_{\{0\}}(X^+_s) ds
\end{align*}
properly defined as a limit on the interval $[0-\epsilon, 0+\epsilon]$ as $\epsilon \to 0$. We also recall that a positive continuous additive functional $A=\{A_t\}_{t\geq 0}$ for $X^+$ is in Revuz correspondence with a measure $\mu_A$, that is 
\begin{align*}
\lim_{t \to 0} \mathbf{E}_m \left[ \frac{1}{t} \int_0^t f(X^+_s) dA_s \right] = \int_\Omega f(x)\mu_A(dx)
\end{align*}
where $f$ is Borel measurable on $\Omega$ and $\mathbf{E}_m = \int \mathbf{E}_x m(dx)$.
Moreover, for $x \in \overline{\Omega}=[0, \infty)$,
\begin{align*}
\mathbf{E}_x \left[ \int_0^t f(X^+_s) dA_s \right] = \int_0^t \int_\Omega f(y) p^+(s,x,y)\mu_A(dy) ds.
\end{align*}
In particular, we observe that given $A$ for $X$ with $supp[\mu_A]=\Lambda \subset [0, \infty)$ we have
\begin{align*}
\lim_{t\to 0} \mathbf{E}_m \left[ \frac{1}{t} \int_0^t f(X_s) dA_s\right] = \int_\Lambda f(x)dx + (\eta/\sigma) \int_\Lambda f(x) \delta_0(dx).
\end{align*}

We provide the following result for the process $\bar{X}$ on $[0, \infty)$, that is an elastic Brownian motion on $[0, \infty)$ running with the new clock $\bar{V}_t^{-1}$. Let us introduce the space
\begin{align*}
D_L = & \bigg\{ u: (0, \infty) \times [0, \infty) \to \mathbb{R} \textrm{ such that } \dot{u}(s,0) (t-s)^{-\alpha} \in L^1(0,t), \, t>s>0 \bigg\}.
\end{align*}
 
\begin{theorem}
The solution $u \in C^{1,2}((0, \infty) \times [0, \infty)) \cap D_L$ to
\begin{align}
\label{FBVP-realLine}
\left\lbrace
\begin{array}{ll}
\displaystyle \dot{u}(t,x) = u^{\prime \prime}(t,x), \quad t>0, \, x >0\\
\\
\displaystyle \eta D^\alpha_t u(t,0) = \sigma u^\prime(t,0) - c\, u(t, 0), \quad t>0\\
\\
\displaystyle u(0,x) = f(x), \quad x \geq 0, \quad f \in C_b[0, \infty)
\end{array}
\right.
\end{align}
has the probabilistic representation
\begin{align*}
u(t,x) = \mathbf{E}_x [f(\bar{X}_t)] = \mathbf{E}_x\left[f(X^+ \circ \bar{V}^{-1}_t) \exp \left( -c/\sigma\, \gamma^+ \circ \bar{V}^{-1}_t \right) \right]
\end{align*}
where $\bar{V}^{-1}_t$ is the inverse of the process
\begin{align*}
\bar{V}_t = t + (\eta/ \sigma)^{1/\alpha} H \circ \gamma^+_t.
\end{align*}
Moreover, 
\begin{align*}
\int_0^\infty e^{-\lambda t} u(t,x) dt = R^\dagger_\lambda f(x) + \frac{e^{-x \sqrt{\lambda}}}{c + \eta \lambda^\alpha + \sigma \sqrt{\lambda}} \left( \sigma \int_0^\infty e^{-y \sqrt{\lambda}} f(y) dy + \eta \frac{\lambda^\alpha}{\lambda}  f(0) \right)
\end{align*}
with $\lambda>0$, where $R^\dagger_\lambda$ has been defined in \eqref{potXdef}.
\label{thm:FBVP-realLine}
\end{theorem}

\begin{proof}
	Since $u$ solves the heat equation, then $u$ can be written as
	\begin{align}
	\label{abWITHsol}
	u(t,x) = a(t,x) + \int_0^t b(t-s, x) u(s,0)ds
	\end{align}
	with Laplace transform
	\begin{align*}
	\tilde{u}(\lambda, x) = \tilde{a}(\lambda, x) + \tilde{b}(\lambda, x)\, \tilde{\upsilon}(\lambda, 0), \quad \lambda>0
	\end{align*}
	for some sufficiently regular and (piecewise continuous) integrable functions $a,b$. After some calculation we arrive at $\tilde{a}(\lambda, 0)=0$ and $\tilde{b}(\lambda, 0)=1$. Moreover, we must have $u(0,x) = a(0, x)$ and $\tilde{u}^{\prime \prime} = \lambda \tilde{u} - f$. In particular, it turns out that $a,b$ are given by
\begin{align}
\label{genREPu}
u(t,x) = Q^\dagger_t f(x) + \int_0^t \frac{x}{s} g(s,x)\, u(t-s, 0)\, ds 
\end{align}
where
\begin{align}
\label{semigQD}
Q^\dagger_t f(x) = \int_0^\infty \big( g(t,x-y) - g(t, x+y) \big)\, f(y)\, dy
\end{align}
is the Dirichlet semigroup and $g(t,z)= e^{-z^2/4t} / \sqrt{4\pi t}$ is the Gaussian kernel. The associated $\lambda$-potential is given by
\begin{align*}
\tilde{u}(\lambda, x) 
= & \int_0^\infty e^{-\lambda t} u(t,x) dt\\
= & \int_0^\infty e^{-\lambda t} Q^\dagger_t f(x) dt + e^{-x \sqrt{\lambda}} \tilde{u}(\lambda, 0), \quad \lambda>0
\end{align*}
where the potential of $Q^\dagger_t$ can be written as
\begin{align}
\int_0^\infty e^{-\lambda t} Q^\dagger_t f(x)\, dt
= & \frac{1}{2} \int_0^\infty \left( \frac{e^{-|x-y|\sqrt{\lambda}}}{\sqrt{\lambda}} - \frac{e^{-(x+y)\sqrt{\lambda}}}{\sqrt{\lambda}} \right) f(y)\, dy \label{PotSemigDir} \\
= & \frac{1}{2} \int_0^\infty \left( \frac{e^{(x-y)\sqrt{\lambda}}}{\sqrt{\lambda}}  - \frac{e^{-(x+y)\sqrt{\lambda}}}{\sqrt{\lambda}} \right) f(y)\, dy \notag \\ 
& - \frac{1}{2} \int_0^x \left( \frac{e^{(x-y)\sqrt{\lambda}}}{\sqrt{\lambda}} - \frac{e^{-(x-y)\sqrt{\lambda}}}{\sqrt{\lambda}} \right) f(y)\, dy, \quad \lambda>0 \notag
\end{align}
by exploiting the fact that
\begin{align*}
& \int_0^\infty e^{-|x-y|} f(y)\, dy \notag \\
= & \int_0^x e^{-(x-y)} f(y)\, dy + \int_x^\infty e^{-(y-x)} f(y)\, dy\notag \\
= & \int_0^\infty e^{-(y-x)}f(y)\, dy + \int_0^x e^{-(x-y)} f(y)\, dy  - \int_0^x e^{-(y-x)} f(y)\, dy.
\end{align*}
We can immediately see that
\begin{align*}
\frac{\partial^2 \tilde{u}}{\partial x^2}(\lambda, x) = \lambda \tilde{u}(\lambda, x) - f(x)
\end{align*}
and
\begin{align}
\label{dertildeu}
\frac{\partial \tilde{u}}{\partial x}(\lambda, x) \bigg|_{x=0} = \int_0^\infty e^{-y \sqrt{\lambda}} f(y) - \sqrt{\lambda} \tilde{u}(\lambda, x).
\end{align}

Let us recall that $\bar{\tau}_0 = \tau_0$ is the first time the process $\bar{X}$ hits the boundary point $x=0$. The representation \eqref{genREPu} can be obtained from the fact that $\bar{X}=X$ and therefore enjoys the Markov property up to the time $\bar{\tau}_0$. In particular, we have that
\begin{align*}
&\mathbf{E}_x\left[ \int_0^\infty e^{-\lambda t} f(\bar{X}_t) dt \right]\\
= & \mathbf{E}_x \left[ \int_0^{\bar{\tau}_0} e^{-\lambda t} f(\bar{X}_t) dt \right] + \mathbf{E}_x [e^{-\lambda \bar{\tau}_0}] \, \mathbf{E}_{0} \left[ \int_0^\infty e^{-\lambda t} f(\bar{X}_t) dt \right] , \quad \lambda>0.
\end{align*}
Let us focus on
\begin{align}
\mathbf{E}_{0} \left[ \int_0^\infty e^{-\lambda t} f(\bar{X}_t) dt \right] 
= & \mathbf{E}_{0} \left[ \int_0^\infty e^{-\lambda t} f(X^+ \circ \bar{V}^{-1}_t) \exp\left( -c/\sigma\, \gamma^+ \circ \bar{V}^{-1}_t \right) dt \right] \notag \\
= &  \mathbf{E}_{0} \left[ \int_0^\infty e^{-\lambda \bar{V}_t} f(X^+_t) \exp\left( -c/\sigma\, \gamma^+_t \right) d\bar{V}_t \right], \quad \lambda>0 \label{potXbarProof}
\end{align}
where
\begin{align*}
\int_0^\infty e^{-\lambda t} \mathbf{P}_0(X^+_t \in dy, 0< \gamma^+_t \in dw) = e^{-(y+w) \sqrt{\lambda}} dy\, dw, \quad \lambda>0 .
\end{align*}
Based on the definition of $\bar{V}_t$, we write \eqref{potXbarProof} as the sum of
\begin{align*}
\mathbf{E}_{0} \left[ \int_0^\infty e^{-\lambda t} f(X^+_t) \exp\left(-\lambda (\eta/\sigma)^{1/\alpha} H\circ \gamma^+_t - c / \sigma \, \gamma^+_t \right) dt \right]
\end{align*}
and
\begin{align*}
(\eta/\sigma)^{1/\alpha} \mathbf{E}_{0} \left[ \int_0^\infty e^{-\lambda t} f(X^+_t) \exp\left(-\lambda (\eta/\sigma)^{1/\alpha} H\circ \gamma^+_t - c/\sigma \, \gamma^+_t \right) d(H \circ \gamma^+_t) \right].
\end{align*}
From the first integral we get
\begin{align*}
& \mathbf{E}_{0} \left[ \int_0^\infty e^{-\lambda t} f(X^+_t) \exp\left(- \lambda (\eta/\sigma)^{1/\alpha}  H \circ \gamma^+_t - c/\sigma\, \gamma^+_t \right) dt \right] \\
= & \mathbf{E}_{0} \left[ \int_0^\infty e^{-\lambda t} f(X^+_t) \exp\left(- \lambda^\alpha \eta/\sigma \, \gamma^+_t - c/\sigma\, \gamma^+_t \right) dt \right] \\
= & \int_0^\infty \int_0^\infty f(y)\, e^{- w \,c/\sigma - w \lambda^\alpha \eta /\sigma} e^{-(y+w) \sqrt{\lambda}} dw\, dy\\
= & \frac{1}{c / \sigma  + \lambda^\alpha \eta /\sigma + \sqrt{\lambda}} \int_0^\infty e^{-y \sqrt{\lambda}} f(y)dy\\
= & \frac{\sigma}{c  + \lambda^\alpha \eta + \sigma \sqrt{\lambda}} \int_0^\infty e^{-y \sqrt{\lambda}} f(y)dy, \quad \lambda>0.
\end{align*}
Set
\begin{align*}
I_3:= \frac{\sigma}{c  + \lambda^\alpha \eta + \sigma \sqrt{\lambda}} \int_0^\infty e^{-y \sqrt{\lambda}} f(y)dy, \quad \lambda>0.
\end{align*}
Now, observe that
\begin{align*}
\mathbf{E}\left[ e^{-\theta H_t} dH_t \right] = - \frac{1}{\theta} d \mathbf{E} \left[ e^{-\theta H_t}\right] = \frac{\theta^\alpha}{\theta} e^{- \theta^\alpha t} dt, \quad \theta>0
\end{align*}
and write (recall that $\gamma^+ \circ \gamma^{+,-1}_t = t$ a.s.)
\begin{align*}
& \mathbf{E}_{0} \left[ \int_0^\infty e^{-\lambda t} f(X^+_t) \exp\left(-\lambda (\eta/\sigma)^{1/\alpha} H\circ \gamma^+_t - c/\sigma \, \gamma^+_t \right) d(H \circ \gamma^+_t) \right]\\
= & \mathbf{E}_{0} \left[ \int_0^\infty e^{-\lambda \gamma^{-1}_t} f(X^+ \circ \gamma^{+,-1}_t) \exp\left(-\lambda (\eta/\sigma)^{1/\alpha} H_t - c/ \sigma \, t \right) dH_t \right]\\
= & \frac{\sigma^{1/\alpha}}{\eta^{1/\alpha}} \frac{\lambda^\alpha \eta}{\lambda \sigma} \mathbf{E}_0 \left[ \int_0^\infty e^{-t \lambda^\alpha \eta /\sigma }  e^{-\lambda \gamma^{+,-1}_t - c/\sigma \, t} f(X^+ \circ \gamma^{+,-1}_t) dt \right]\\
= & \frac{\sigma^{1/\alpha}}{\eta^{1/\alpha}} \frac{\lambda^\alpha \eta}{\lambda \sigma} \mathbf{E}_0 \left[ \int_0^\infty  e^{-\lambda t} f(X^+_t) \exp \left( -\lambda^\alpha \eta /\sigma \, \gamma^+_t - c/\sigma \, \gamma^+_t \right) d\gamma^+_t \right], \quad \lambda>0.
\end{align*}
Set 
\begin{align*}
I_4:= \frac{\sigma^{1/\alpha}}{\eta^{1/\alpha}} \frac{\lambda^\alpha \eta}{\lambda \sigma} \mathbf{E}_0 \left[ \int_0^\infty  e^{-\lambda t} f(X^+_t) \exp \left( -\lambda^\alpha \eta /\sigma \, \gamma^+_t - c/\sigma \, \gamma^+_t \right) d\gamma^+_t \right], \quad \lambda>0
\end{align*}
where
\begin{align*}
& \mathbf{E}_0 \left[ \int_0^\infty  e^{-\lambda t} f(X^+_t) \exp \left( -\lambda^\alpha \eta / \sigma \, \gamma^+_t - c/\sigma\, \gamma^+_t \right) d\gamma^+_t \right] \\
= & \frac{f(0)}{c / \sigma + \lambda^\alpha \eta /\sigma} \mathbf{E}_0 \left[1 - \lambda \int_0^\infty e^{-\lambda t} \exp \left( - \lambda^\alpha \eta / \sigma \, \gamma^+_t  - c/\sigma\, \gamma^+_t \right) dt  \right]\\
= & \frac{f(0)}{c/\sigma + \lambda^\alpha \eta /\sigma } \left( 1 - \sqrt{\lambda} \int_0^\infty e^{-w \, c/\sigma - w\, \lambda^\alpha \eta /\sigma - w\sqrt{\lambda} } dw \right)\\
= & \frac{f(0)}{c/\sigma  + \lambda^\alpha \eta / \sigma + \sqrt{\lambda}}\\
= & \frac{\sigma\, f(0)}{c  + \lambda^\alpha \eta + \sigma \sqrt{\lambda}}, \quad \lambda>0.
\end{align*}
That is,
\begin{align*}
I_4 = \frac{\sigma^{1/\alpha}}{\eta^{1/\alpha}} \frac{\lambda^\alpha}{\lambda} \frac{\eta f(0)}{c + \eta \lambda^\alpha  + \sigma \sqrt{\lambda}}, \quad \lambda>0.
\end{align*}
Then, we obtain
\begin{align}
\label{uLapSolZero}
\tilde{u}(\lambda, 0) = & I_3 + (\eta/\sigma)^{1/\alpha} I_4 \notag \\ 
= & \frac{\sigma}{c + \eta \lambda^\alpha + \sigma \sqrt{\lambda}} \int_0^\infty e^{-y \sqrt{\lambda}} f(y) dy + \frac{\lambda^\alpha}{\lambda} \frac{\eta f(0)}{c + \eta \lambda^\alpha + \sigma \sqrt{\lambda}}
\end{align}
from which we get that
\begin{align*}
\eta \lambda^\alpha \tilde{u} - \eta \frac{\lambda^\alpha}{\lambda} f(0) = \sigma \int_0^\infty e^{-y \sqrt{\lambda}} f(y) dy -  \sigma \sqrt{\lambda} \tilde{u} - c\, \tilde{u}.
\end{align*}
By observing that
\begin{align*}
\eta \lambda^\alpha \tilde{u} - \eta \frac{\lambda^\alpha}{\lambda} f(0) = \eta \int_0^\infty e^{-\lambda t} D^\alpha_t u(t,0)\, dt,
\end{align*}
from \eqref{dertildeu}, we obtain
\begin{align*}
\eta D^\alpha_t u(t,0) = \sigma \frac{\partial u}{\partial x}(t, 0) - c\, u(t, 0)
\end{align*}
which is the claimed boundary condition. \\

The requirement $u \in D_L$ ensures existence of $D^\alpha_t u(t, 0)$. Uniqueness follows from the continuity of $u$ and the Laplace transform machinery (Lerch's Theorem). \\

\end{proof}
For the reader's convenience we recall that
\begin{align*}
\tilde{u}(\lambda, 0) = \int_0^\infty e^{-\lambda t} u(t,0) dt
\end{align*}
and $u$ is the solution to the problem \eqref{FBVP-realLine}. Now we study
\begin{align*}
\lambda \tilde{u}(\lambda, 0) - f(0) = \int_0^\infty \lambda  e^{-\lambda t} \mathbf{E}_0[f(\bar{X}_t) - f(\bar{X}_0)] dt = \mathbf{E}_0 [f(\bar{X}_\chi) - f(\bar{X}_0)]
\end{align*}
where $\chi$ is an exponential r.v. with $\mathbf{P}(\chi > t) = e^{-\lambda t}$, $\lambda>0$. Denote by $H^{1/2}$ the stable subordinator of order $1/2$ independent from $H$. Recall that $H$ is a stable subordinator of order $\alpha \in (0,1)$.

\begin{lemma}
If $\| f\|_\infty= f(0)$, then
\begin{align}
\lambda \tilde{u}(\lambda,0) - f(0) \leq -  f(0)\, \mathbf{E} \left[ \exp \left( - \lambda H^{1/2}_{\sigma \chi} - \lambda H_{\eta \chi} \right) \right], \quad \lambda \geq 0
\label{estMiddleSharp}
\end{align}
with $\mathbf{P}(\chi >t) = e^{-ct}$ and
\begin{align}
\lambda \tilde{u}(\lambda, 0) - f(0) \geq - f(0) \, \sqrt{(c/\eta)\frac{1}{\lambda^\alpha}+ (\sigma/\eta)\frac{\sqrt{\lambda}}{\lambda^\alpha}}, \quad \lambda \geq 0.
\label{estBelow}
\end{align}
Otherwise, 
\begin{align}
\lambda \tilde{u}(\lambda, 0) - f(0) \leq \frac{1}{2}\big( \| f \|_\infty - f(0)\big) \sqrt{(c/\eta)\frac{1}{\lambda^\alpha}+ (\sigma/\eta)\frac{\sqrt{\lambda}}{\lambda^\alpha}} , \quad \lambda > 0
\label{estMiddle}
\end{align}
and
\begin{align}
\lim_{\lambda \to 0} \lambda \tilde{u}(\lambda, 0) - f(0) \leq \big( \| f \|_\infty - f(0)\big).
\label{estMiddleZero}
\end{align}
Moreover,
\begin{align}
| \lambda \tilde{u}(\lambda, 0) - f(0) | \leq  \| f \|_\infty \, \sqrt{(c/\eta)\frac{1}{\lambda^\alpha}+ (\sigma/\eta)\frac{\sqrt{\lambda}}{\lambda^\alpha}}, \quad \lambda > 0.
\label{estModulus}
\end{align}

\end{lemma}
\begin{proof}
We use the fact that, for $\lambda>0$,
\begin{align*}
\int_0^\infty e^{-\lambda t} \frac{\partial u}{\partial t} (t, 0) dt \leq  \int_0^\infty e^{-\lambda t} \varphi(t) dt \quad \textrm{iff} \quad \frac{\partial u}{\partial t} (t, 0) \leq \varphi(t)
\end{align*}
for a given function $\varphi: (0, \infty) \mapsto (0, \infty)$ for which the Laplace transform exists. We observe that, from \eqref{uLapSolZero}, 
\begin{align}
\lambda \tilde{u}(\lambda, 0) - f(0) 
\leq & \frac{\sigma \sqrt{\lambda}}{c + \lambda^\alpha \eta + \sigma \sqrt{\lambda}} \|f\|_\infty + \frac{\eta \lambda^\alpha}{c + \eta \lambda^\alpha + \sigma \sqrt{\lambda}} f(0) - f(0) \label{tmpMLuZero}\\
= & \frac{\sigma \sqrt{\lambda}}{c + \lambda^\alpha \eta + \sigma \sqrt{\lambda}} \|f\|_\infty - \frac{c + \sigma \sqrt{\lambda}}{c + \eta \lambda^\alpha + \sigma \sqrt{\lambda}}f(0)\notag  \\
\leq & \big( \| f \|_\infty - f(0)\big)  \frac{c + \sigma \sqrt{\lambda}}{c + \eta \lambda^\alpha + \sigma \sqrt{\lambda}}, \quad \lambda \geq 0. \notag
\end{align}
This leads to \eqref{estMiddleZero}. By exploiting the Young's inequality  $a+b \geq 2 \sqrt{ab}$ for $a=c+\sigma\sqrt{\lambda}$ and $b=\eta \lambda^\alpha$, we obtain
\begin{align}
\frac{c + \sigma \sqrt{\lambda}}{c + \eta \lambda^\alpha + \sigma \sqrt{\lambda}} \leq \frac{1}{2} \sqrt{(c/\eta)\frac{1}{\lambda^\alpha}+ (\sigma/\eta)\frac{\sqrt{\lambda}}{\lambda^\alpha}} 
\end{align}
and this proves \eqref{estMiddle}. Since $-f \leq |f|$ we also get
\begin{align*}
\big( \| f \|_\infty - f(0)\big) \leq \big( \| f \|_\infty + |f(0)|\big) \leq 2 \|f\|_\infty
\end{align*}
and
\begin{align*}
\lambda \tilde{u}(\lambda, 0) - f(0) \leq  \| f \|_\infty \, \sqrt{(c/\eta)\frac{1}{\lambda^\alpha}+ (\sigma/\eta)\frac{\sqrt{\lambda}}{\lambda^\alpha}}, \quad \lambda \geq 0.
\end{align*} 
Moreover,
\begin{align*}
f(0) - \lambda \tilde{u}(\lambda, 0) 
= & \frac{c + \sigma \sqrt{\lambda}}{c + \eta \lambda^\alpha + \sigma \sqrt{\lambda}} f(0) - \frac{\sigma \lambda}{c + \lambda^\alpha \eta + \sigma \sqrt{\lambda}} \int_0^\infty e^{-y \sqrt{\lambda}} f(y) dy\\
\leq & \frac{c + \sigma \sqrt{\lambda}}{c + \eta \lambda^\alpha + \sigma \sqrt{\lambda}} f(0) + \frac{\sigma \lambda}{c + \lambda^\alpha \eta + \sigma \sqrt{\lambda}} \int_0^\infty e^{-y \sqrt{\lambda}} |f(y)| dy\\
\leq & \frac{c + \sigma \sqrt{\lambda}}{c + \eta \lambda^\alpha + \sigma \sqrt{\lambda}} f(0) + \frac{\sigma \sqrt{\lambda}}{c + \lambda^\alpha \eta + \sigma \sqrt{\lambda}} \| f\|_\infty\\
\leq & \frac{c + \sigma \sqrt{\lambda}}{c + \eta \lambda^\alpha + \sigma \sqrt{\lambda}} \big( f(0) + \| f\|_\infty \big)
\end{align*}
where
\begin{align*}
\big( f(0) + \| f\|_\infty \big) \leq 2 \|f \|_\infty.
\end{align*}
Thus,
\begin{align*}
f(0) - \lambda \tilde{u}(\lambda, 0) \leq \| f \|_\infty \, \sqrt{(c/\eta)\frac{1}{\lambda^\alpha}+ (\sigma/\eta)\frac{\sqrt{\lambda}}{\lambda^\alpha}}, \quad \lambda >0
\end{align*}
says that
\begin{align*}
\lambda \tilde{u}(\lambda, 0) - f(0) \geq - \| f \|_\infty \, \sqrt{(c/\eta)\frac{1}{\lambda^\alpha}+ (\sigma/\eta)\frac{\sqrt{\lambda}}{\lambda^\alpha}}, \quad \lambda >0
\end{align*}
which is \eqref{estBelow}. We conclude that
\begin{align*}
| \lambda \tilde{u}(\lambda, 0) - f(0) | \leq  \| f \|_\infty \, \sqrt{(c/\eta)\frac{1}{\lambda^\alpha}+ (\sigma/\eta)\frac{\sqrt{\lambda}}{\lambda^\alpha}}, \quad \lambda > 0
\end{align*}
as stated in \eqref{estModulus}.\\

Now we use \eqref{uLapSolZero} once again and write
\begin{align}
\lambda \tilde{u}(\lambda,0) - f(0)
\leq & \frac{\sigma \sqrt{\lambda}}{c + \lambda^\alpha \eta + \sigma \sqrt{\lambda}} \| f\|_\infty - \frac{(c + \sigma \sqrt{\lambda})}{c + \eta \lambda^\alpha + \sigma \sqrt{\lambda}} f(0) \notag \\
= & \big( \| f \|_\infty - f(0)\big) \frac{\sigma \sqrt{\lambda}}{c + \lambda^\alpha \eta + \sigma \sqrt{\lambda}} - \frac{c}{c + \eta \lambda^\alpha + \sigma \sqrt{\lambda}} f(0) . \label{estMproof}
\end{align}
In case $\| f\|_\infty = f(0)$ we get
\begin{align}
\lambda \tilde{u}(\lambda,0) - f(0)
\leq - \frac{c}{c + \eta \lambda^\alpha + \sigma \sqrt{\lambda}} f(0).
\label{tmpNeg}
\end{align}
Thus, by observing that
\begin{align*}
\mathbf{E} e^{-\theta H_t} = e^{-t \theta^\alpha}, 
\end{align*}
from 
\begin{align*}
\frac{c}{c + \eta \lambda^\alpha + \sigma \sqrt{\lambda}} = c \int_0^\infty e^{-c t} e^{- t \sigma \sqrt{\lambda} - t \eta \lambda^\alpha} dt
\end{align*}
we prove \eqref{estMiddleSharp}.
\end{proof}

\begin{remark}
We argue on $u \in D_L$. This ensures that
\begin{align*}
\big| D^\alpha_t u(t,0) \big| \leq \frac{1}{\Gamma(1-\alpha)} \int_0^t \left|\frac{\partial u}{\partial s}(s,0) \right| (t-s)^{-\alpha} ds  < \infty.
\end{align*}

If $\sigma=0$ in \eqref{estMiddle}, then
\begin{align*}
\sqrt{c/\eta} \frac{1}{\lambda^{\alpha/2}} = \sqrt{c/\eta} \frac{1}{\Gamma(\alpha/2)}  \int_0^\infty e^{-\lambda t} t^{\alpha/2 -1} dt
\end{align*}
and $\forall\, \alpha/2>0$,
\begin{align*}
\int_0^t \frac{\partial u}{\partial s}(s,0) (t-s)^{-\alpha} ds \leq (\|f\|_\infty  - f(0)) \int_0^t  t^{\alpha/2 -1} (t-s)^{-\alpha} ds < \infty. 
\end{align*}
Moreover, from \eqref{estModulus} and the fact that
\begin{align*}
\bigg| \frac{\partial u}{\partial t}(t, 0) \bigg| < \infty\;\; a.e.
\end{align*}
we conclude that
\begin{align}
\exists\, \mathcal{C}>0\,:\,  \bigg| \frac{\partial u}{\partial t}(t, 0) \bigg| \leq \mathcal{C}\,  \| f\|_\infty\, \frac{\sqrt{(c/\eta)}}{2}  t^{\alpha/2-1}, \quad t>0.
\label{estMmodulus}
\end{align}
We notice that $\sigma=0$ is strictly related with $f(0)=\|f\|_\infty$ as formula \eqref{estMproof} entails. In this regard, we notice that $f\geq 0$ and formula \eqref{estMiddleSharp} says also that
\begin{align*}
\tilde{u}(\lambda, 0) \leq f(0) \frac{1}{\lambda} \mathbf{E} \left[ 1 - \exp(- \lambda H^{1/2}_{\sigma \chi} - \lambda H_{\eta \chi}) \right], \quad 0 \leq \sigma < \infty, \; 0 \leq \eta < \infty
\end{align*}
where $\chi$ is an independent exponential r.v. with parameter $c>0$. Observe that $\sigma=\eta=0$ gives the Dirichlet boundary condition. For $\sigma=0$ we have 
\begin{align*}
\tilde{u}(\lambda, 0) \leq f(0) \frac{1}{\lambda}\mathbf{E} \left[ 1 - \exp(- \lambda H \circ e_0) \right], \quad 0 \leq \eta < \infty
\end{align*}
where $e_0$ is the holding time (at zero) for $X$ (see also Theorem \ref{thm:holdingTime} below). \\

If $c=0$, from \eqref{estMiddle} we write
\begin{align*}
\frac{\partial u}{\partial t}(t, 0) \leq \varphi(t) \quad \textrm{with} \quad \int_0^\infty e^{-\lambda t} \varphi(t)dt = C\, \lambda^{1/4 - \alpha/2}, \quad \lambda >0
\end{align*}
for some constant $C\leq (\|f\|_\infty - f(0)) \sqrt{\sigma/\eta}$. Thus, 
\begin{align*}
D^\alpha_t u(t,0) \leq \frac{1}{\Gamma(1-\alpha)} \int_0^t \varphi(s) (t-s)^{-\alpha} ds
\end{align*}
and 
\begin{align*}
\int_0^\infty e^{-\lambda t} D^\alpha_t u(t,0) \,dt \leq  \frac{\lambda^\alpha}{\lambda} \,C \,\lambda^{1/4 - \alpha/2} = C\, \lambda^{1/4 + \alpha/2 - 1}.
\end{align*}
We conclude that 
\begin{align*}
D^\alpha_t u(t,0) \leq C \, t^{(1 - 1/4 - \alpha/2)-1}, \quad t>0
\end{align*}
and $u \in D_L$.
\end{remark}

\begin{remark}
Finally we focus on \eqref{tmpMLuZero} with $\sigma=0$. The inequality
\begin{align*}
\lambda \tilde{u}(\lambda, 0) 
\leq & f(0)\, \frac{\eta \lambda^\alpha}{c + \eta \lambda^\alpha}
\end{align*}
implies that
\begin{align*}
\tilde{u}(\lambda, 0) \leq & f(0) \int_0^\infty e^{-\lambda t} E_\alpha (-(c/\eta) t^\alpha)\, dt, \quad \lambda>0.
\end{align*}
That is,
\begin{align*}
u(t,0) \leq f(0)\, E_\alpha (-(c/\eta) t^\alpha) = f(0)\, \mathbf{P}(\chi >t) 
\end{align*}
where $\chi$ denotes a Mittag-Leffler random variable with parameters $\alpha$ and $(c/\eta)$.
\end{remark}

\begin{figure}
\centering
\includegraphics[scale=.5]{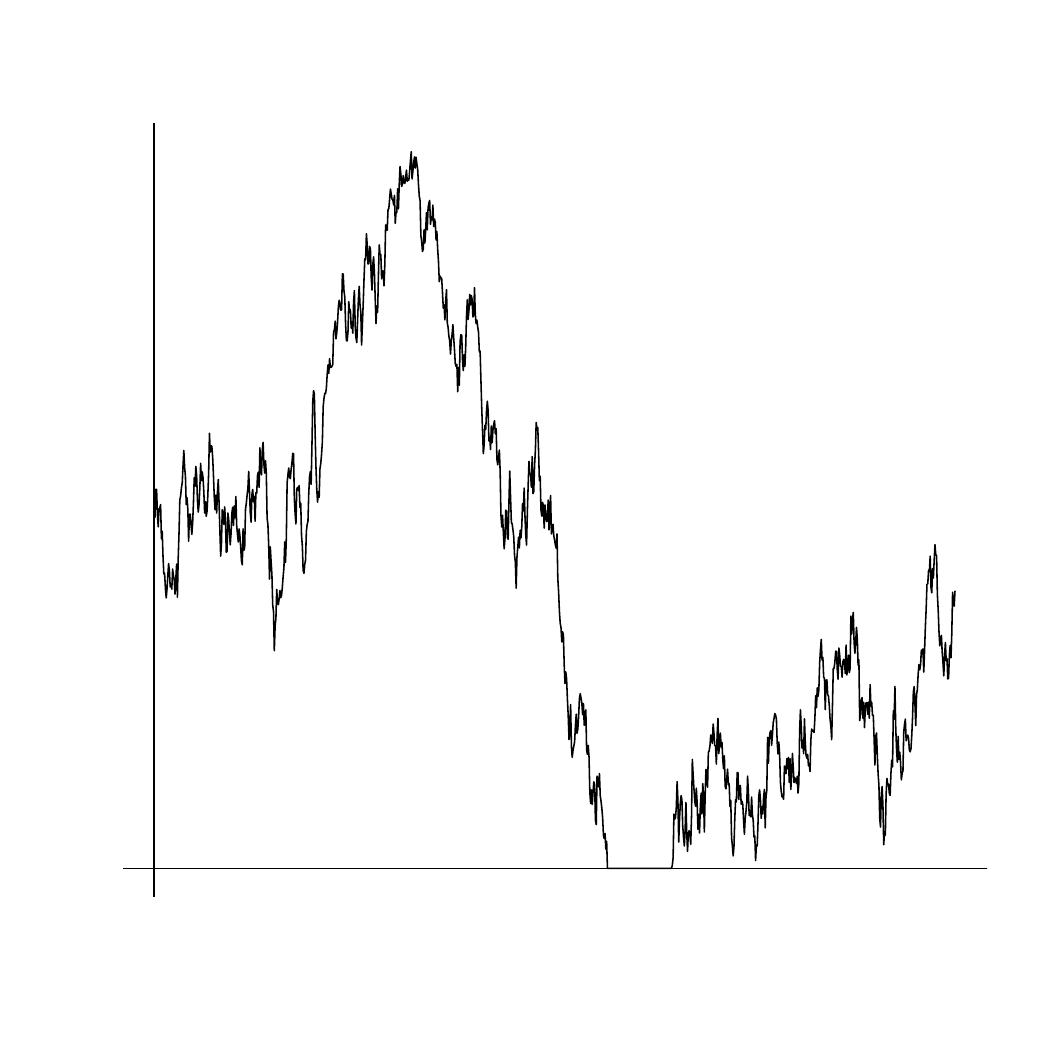} 
\caption{A representation of $\bar{X}$ on $[0, \infty)$ as a motion on the path of $X^+$ (a Brownian motion reflected at zero). The plateau is given by the inverse of $\bar{V}_t = t + H \circ (\eta/\sigma) \gamma^+_t$. As the local time at zero $\gamma^+$ increases the jump of $H$ produces a plateau for $\bar{V}^{-1}_t$. According with this plateau, the process $X^+ \circ \bar{V}^{-1}_t$ spends more time on the boundary point $\{0\}$. The path exhibits a delayed reflection. The delay is given by $H$ which is independent from $X^+$. This delay is the holding time with Mittag-Leffler distribution.}
\label{Fig:barX}
\end{figure}

\subsection{On the occupation times for $\bar{X}$}

We recall that $\{t\,:\, X^+_t=0\}$ is a perfect closed set of zero Lebesgue measure. The definition of $\gamma^+$ can be therefore given with some care as a proper limit which holds a.s. and $L^2(0,\infty)$. On the other hand $\{t\,:\, X_t=0\}$ has positive Lebesgue measure as well as $\{t\,:\, \bar{X}_t=0\}$. For $\alpha \in (0,1]$ and $\epsilon > 0$ we can write 
\begin{align}
\mathbf{P}_0(0 \leq \bar{X}_t \leq \epsilon) = \int_0^\epsilon \bar{p}(t,0,y)dy + (\eta/\sigma) \bar{p}(t,0,0)
\label{probIntZero}
\end{align}
whose potential writes
\begin{align}
\label{potZeroEpsilon}
\mathbf{E}_0\left[ \int_0^\infty e^{-\lambda t} \mathbf{1}_{[0, \epsilon]}(\bar{X}_t) dt  \right] 
= & \frac{\sigma}{c + \lambda^\alpha \eta + \sigma \sqrt{\lambda}} \frac{1 - e^{-\epsilon \sqrt{\lambda}}}{\sqrt{\lambda}} + \frac{\lambda^\alpha}{\lambda} \frac{\eta}{c + \eta \lambda^\alpha + \sigma \sqrt{\lambda}}
\end{align}
as obtained in \eqref{uLapSolZero}. Notice that both limits $\epsilon \to 0$ and $\epsilon \to \infty$ make sense and 
\begin{align*}
\int_0^\infty e^{-\lambda t} \bar{p}(t,0,0) dt = \frac{\lambda^\alpha}{\lambda} \frac{\sigma}{c + \eta \lambda^\alpha + \sigma \sqrt{\lambda}}, \quad \lambda>0
\end{align*}
becomes
\begin{align*}
\int_0^\infty e^{-\lambda t} p(t,0,0) dt = \frac{\sigma}{c + \eta \lambda + \sigma \sqrt{\lambda}}, \quad \lambda>0
\end{align*}
for $\alpha=1$. We immediately see that
\begin{align*}
\int_0^\infty p(t,0,0) dt = \sigma/c \quad \textrm{and} \quad (\eta/\sigma) \int_0^\infty p(t,0,0) dt = \eta/c
\end{align*}
give interesting readings of the parameters $\eta, \sigma, c$ and their asymptotic analysis. 

We also observe that simple manipulation leads to
\begin{align*}
\bar{p}(t,0,0) = \int_0^t \frac{\sigma s}{t-z} g(t-z, \sigma s) \, \eta \,l(z, \eta s) dz
\end{align*}
where $g$ and $l$ have been respectively defined in pages \pageref{semigQD} and \pageref{MLbound}. Explicit representations of $l$ are well-known for some values of $\alpha$. For example:
\begin{itemize}
\item[i)] $\alpha=1/2$,
\begin{align*}
l(t,x) = 2 g(t,x) = 2 \frac{e^{-\frac{x^2}{4t}}}{\sqrt{4\pi t}}, \quad t>0,\, x \in (0, \infty),
\end{align*}
\item[ii)] $\alpha=1/3$,
\begin{align*}
l(t,x) = \frac{3}{\sqrt[3]{3t}}Ai\left( \frac{x}{\sqrt[3]{3t}} \right), \quad t>0,\, x \in (0, \infty)
\end{align*}
where $Ai$ is the Airy function. 
\end{itemize}
The reader can consult \cite{OrsDov2012} and the references therein for further details.\\

Now we recall that
\begin{align*}
f(\bar{X}_t) = f(X^+ \circ \bar{V}^{-1}_t) \, M \circ \bar{V}^{-1}_t
\end{align*}
where $M \leq 1$ is the multiplicative functional associated with the Robin boundary condition. If we assume $c=0$, then $\bar{X}$ has no elastic kill and we write
\begin{align*}
f(\bar{X}_t) = f(X^+ \circ \bar{V}^{-1}_t).
\end{align*} 
We present the following result concerning the processes $\bar{V}_t = t + H \circ (\eta/\sigma) \gamma^+_t$, $t\geq 0$ and the right-inverse $\bar{V}^{-1}_t = \inf\{s \,:\, \bar{V}_s >t\}$, $t \geq 0$. Recall that a.s. $\bar{V}_t \geq t$ and $\bar{V}^{-1}_t \leq t$. In particular,
\begin{align*}
\mathbf{P}(t + H\circ (\eta/\sigma)\gamma^+_t \geq s) > 0 \quad \textrm{for} \quad s \geq t. 
\end{align*}
For $\alpha=1$, $\bar{V}_t = V_t:= t + (\eta/\sigma) \gamma^+_t$ is a continuous process. We introduce the processes 
\begin{align*}
T^H_t = H \circ (\eta/\sigma) \gamma^+_t \quad \textrm{and} \quad T^L_t = \gamma^{+,-1} \circ (\sigma/ \eta) L_t
\end{align*}
such that, for $t>0$, $s>0$,
\begin{align*}
\mathbf{P}_0(T^H_s \geq t) = \mathbf{P}_0(s \geq T^L_t)
\end{align*}
where $\gamma^{+, -1} = \{\gamma^{+,-1}_t\}_{t\geq 0}$ is the inverse of $\gamma^+$. Then, we write
\begin{align*}
\bar{V}_t = t + T^H_t \quad \textrm{and} \quad \bar{\mathcal{V}}_t = t + T^L_t
\end{align*}

\begin{lemma}
The following holds true:
\begin{itemize}
\item[i)] $\mathbf{P}_0(\gamma^+_t \geq s) =\mathbf{P}_0 (\gamma^{+,-1}_s \leq t)$ where $\gamma^{+,-1}$ is a $1/2$-stable subordinator and $s,t>0$;
\item[ii)] $\mathbf{P}_0(\bar{V}_t \geq s) = \mathbf{P}_0(t \geq T^L_{s-t}), \quad 0\leq t\leq s$; 
\item[iii)] $\mathbf{P}_0(\bar{\mathcal{V}}_t \geq s) = \mathbf{P}_0 (t \geq T^H_{s-t}), \quad 0 \leq t \leq s$.
\end{itemize}
\end{lemma}
\begin{proof}
Point $i)$ can be easily verified. Thus we move on point $ii)$. We prove, $\forall\, s>0$, the identity
\begin{align*}
\int_0^s e^{-\xi t} \mathbf{P}_0(t + H\circ (\eta /\sigma)\gamma^+_t \geq s ) dt
= & \int_0^s e^{-\xi t} \mathbf{P}_0(t \geq \gamma^{+, -1} \circ (\sigma/\eta) L_{s-t} ) dt
\end{align*}
with $\xi>0$ from which $\mathbf{P}_0(\bar{V}_t \geq s) = \mathbf{P}_0(t \geq \bar{\mathcal{V}}^{-1}_{s,t})$ for $s\geq t \geq 0$. The left-hand side of the integral above gives
\begin{align*}
& \int_0^\infty e^{-\lambda s} \int_0^s e^{-\xi t} \mathbf{P}_0(t + H\circ (\eta /\sigma)\gamma^+_t \geq s ) dt\, ds \\
= &\int_0^\infty e^{-\xi t} \int_t^\infty e^{-\lambda s} \mathbf{P}_0(t + H\circ (\eta /\sigma)\gamma^+_t \geq s ) ds\, dt\\
= & \frac{1}{\lambda} \int_0^\infty e^{-\xi t}\left(e^{-\lambda t} - \int_t^\infty e^{-\lambda s}  \mathbf{P}_0(t + H\circ (\eta /\sigma)\gamma^+_t \in ds ) \right) dt\\
= & \frac{1}{\lambda (\lambda+\xi)} - \frac{1}{\lambda} \int_0^\infty e^{-\xi t}  \int_t^\infty e^{-\lambda s}  \mathbf{P}_0(t + H\circ (\eta /\sigma)\gamma^+_t \in ds )dt\\
= & \frac{1}{\lambda (\lambda+\xi)} - \frac{1}{\lambda} \int_0^\infty e^{-\xi t} \mathbf{E}_0 \left[ e^{- \lambda t - \lambda H \circ (\eta /\sigma) \gamma^+_t} \right] dt\\
= & \frac{1}{\lambda (\lambda+\xi)} - \frac{1}{\lambda} \int_0^\infty e^{-\xi t} e^{-\lambda t} \mathbf{E}_0[e^{-\lambda^\alpha (\eta / \sigma) \gamma^+_t}] dt\\
= & \frac{1}{\lambda (\lambda+\xi)} - \frac{1}{\lambda} \int_0^\infty (\xi + \lambda)^{-1/2} e^{-w (\xi+\lambda)^{1/2}} e^{-\lambda^\alpha (\eta/\sigma) w} dw\\
= & \frac{1}{\lambda (\lambda+\xi)} - \frac{1}{\lambda} \frac{(\xi + \lambda)^{-1/2}}{(\xi+\lambda)^{1/2} + \lambda^\alpha (\eta/\sigma)}\\
= & \frac{(\xi+\lambda)^{-1/2}}{\lambda} \left( \frac{1}{(\xi+\lambda)^{1/2}} - \frac{1}{(\xi+\lambda)^{1/2} + \lambda^\alpha (\eta/\sigma)}\right)\\
= & \frac{1}{\lambda (\xi + \lambda)} \frac{\lambda^\alpha (\eta/\sigma)}{(\xi+\lambda)^{1/2} + \lambda^\alpha (\eta/\sigma)}, \quad \xi >0, \quad \lambda >0
\end{align*}
which equals   
\begin{align*}
& \int_0^\infty e^{-\lambda s} \int_0^s e^{-\xi t} \mathbf{P}_0(t \geq \gamma^{+, -1} \circ (\sigma/\eta) L_{s-t} ) dt\, ds \\
= & \int_0^\infty e^{-\xi t}  \int_t^\infty e^{-\lambda s}  \mathbf{P}_0(t \geq \gamma^{+, -1} \circ (\sigma/\eta) L_{s-t} ) ds\, dt\\
= & \int_0^\infty e^{-\xi t} \int_0^\infty e^{-\lambda t} e^{-\lambda z}   \mathbf{P}_0(t \geq \gamma^{+, -1} \circ (\sigma/\eta) L_{z} )dz\, dt\\
= & \int_0^\infty e^{-(\xi+\lambda)t} \int_0^\infty e^{-\lambda z} \mathbf{P}_0(t \geq \gamma^{+, -1} \circ (\sigma/\eta) L_{z} )dz\, dt\\
= & \frac{1}{\xi + \lambda} \int_0^\infty e^{-\lambda z} \mathbf{E}_0[e^{-(\xi+\lambda) \gamma^{+,-1} \circ (\sigma/\eta)L_z}] dz\\
= & \frac{1}{\xi + \lambda} \int_0^\infty \mathbf{E}_0[e^{-(\xi+\lambda) \gamma^{+,-1} \circ (\sigma/\eta)w}] \lambda^{\alpha-1} e^{-\lambda^\alpha w} dw\\
= & \frac{1}{\xi + \lambda} \int_0^\infty e^{-(\xi+\lambda)^{1/2} (\sigma/\eta)w} \lambda^{\alpha-1} e^{-\lambda^\alpha w} dw\\
= & \frac{1}{\lambda (\xi + \lambda)}  \frac{\lambda^{\alpha}}{(\xi+\lambda)^{1/2} (\sigma/\eta) +\lambda^\alpha }, \quad \xi>0, \quad \lambda >0
\end{align*}
from the right-hand side. Concerning the point $iii)$ we simply follow the same arguments as before and we evaluate
\begin{align*}
& \int_0^\infty e^{-\xi t} \int_t^\infty e^{-\lambda s} \mathbf{P}_0(t + \gamma^{+, -1} \circ (\sigma/\eta)L_t \geq s) ds\, dt\\
= & \frac{1}{\lambda} \int_0^\infty e^{-\xi t} \left(e^{-\lambda t} - \mathbf{E}_0 [e^{-\lambda t - \lambda \gamma^{+, -1}\circ (\sigma/\eta)L_t}] \right) dt\\
= &  \frac{1}{\lambda (\xi + \lambda)} - \frac{1}{\lambda} \int_0^\infty e^{-\xi t}\mathbf{E}_0 [e^{-\lambda t - \lambda^{1/2} (\sigma/\eta)L_t}] dt\\
= & \frac{1}{\lambda (\xi + \lambda)} - \frac{1}{\lambda} \int_0^\infty e^{-\lambda^{1/2} (\sigma/\eta) w} (\xi + \lambda)^{\alpha-1} e^{-(\xi + \lambda)^\alpha w} dw\\
= & \frac{1}{\lambda (\xi + \lambda)} - \frac{1}{\lambda (\xi+\lambda)} \frac{(\xi+\lambda)^\alpha}{(\xi+\lambda)^\alpha + \lambda^{1/2}(\sigma/\eta)}\\
= & \frac{1}{\lambda (\xi + \lambda)} \frac{\lambda^{1/2} (\sigma/\eta)}{(\xi+\lambda)^\alpha + \lambda^{1/2}(\sigma/\eta)}, \quad \xi>0,\quad \lambda>0.
\end{align*}
Since
\begin{align*}
& \int_0^\infty e^{-\xi t} \int_t^\infty e^{-\lambda s} \mathbf{P}_0 (t \geq T^H_{s-t})\, ds\, dt\\
= & \int_0^\infty e^{-(\xi +\lambda)t} \int_0^\infty e^{-\lambda z} \mathbf{P}_0 (t \geq T^H_{z})\, dz\, dt\\
= & \int_0^\infty e^{-\lambda z} \frac{1}{\xi + \lambda} \mathbf{E}_0 [e^{-(\xi + \lambda) H \circ (\eta/\sigma) \gamma^+_z}] dz\\
= & \frac{1}{\xi + \lambda} \mathbf{E}_0 \left[ \int_0^\infty e^{-(\xi + \lambda) z} e^{-(\xi+\lambda)^\alpha (\eta/\sigma) \gamma^+_z} dz\right]\\
= & \frac{1}{\xi + \lambda} \int_0^\infty e^{-(\xi + \lambda)^\alpha (\eta/\sigma) w} \lambda^{-1/2} e^{-w \lambda^{1/2}} dw\\
= &   \frac{\lambda^{-1/2}}{\xi + \lambda} \frac{1}{(\xi+\lambda)^\alpha (\eta/\sigma) + \lambda^{1/2}}
\end{align*}
we conclude that $iii)$ holds true. 
\end{proof}

For the reader's convenience we report the following identities which are obtained from the potentials above:
\begin{align}
& \int_0^\infty e^{-\lambda s} \int_0^\infty \mathbf{P}(t + H\circ (\eta/\sigma)\gamma^+_t \geq s) dt\, ds\notag \\
= & \lim_{\xi \to 0} \frac{1}{\lambda (\xi + \lambda)} \frac{\lambda^\alpha (\eta/\sigma)}{(\xi+\lambda)^{1/2} + \lambda^\alpha (\eta/\sigma)} \notag \\
= & \frac{1}{\lambda} \frac{\lambda^\alpha}{\lambda} \frac{\eta}{\eta \lambda^\alpha + \sigma \sqrt{\lambda}}, \quad \lambda>0 \label{proofV1}
\end{align}
and
\begin{align}
& \int_0^\infty e^{-\lambda s} \int_0^\infty \mathbf{P}(t + \gamma^{+, -1} \circ (\sigma/\eta) L_t \geq s) dt\, ds \notag \\
= & \lim_{\xi \to 0} \frac{1}{\lambda (\xi + \lambda)} \frac{\lambda^{1/2} (\sigma/\eta)}{(\xi+\lambda)^\alpha + \lambda^{1/2}(\sigma/\eta)}, \quad \lambda>0\notag \\
= & \frac{1}{\lambda} \frac{\sigma}{\sigma \sqrt{\lambda} + \eta \lambda^\alpha} \frac{1}{\sqrt{\lambda}} \label{proofV2}\\
= & \frac{1}{\lambda} \frac{\sigma}{\sigma \sqrt{\lambda} + \eta \lambda^\alpha} \int_0^\infty e^{-y \sqrt{\lambda}} dy. \notag
\end{align}
We can recognize \eqref{proofV1} and \eqref{proofV2} in formula \eqref{uLapSolZero} for $c=0$ and suitable initial data. These formulas will be useful further on. We also observe that
\begin{align*}
& \int_0^\infty e^{-\lambda s} \int_0^s \mathbf{P}_0(t + H\circ (\eta /\sigma)\gamma^+_t \leq s ) dt\, ds \\
= & \lim_{\xi \to 0}\frac{1}{\lambda} \frac{(\xi + \lambda)^{-1/2}}{(\xi+\lambda)^{1/2} + \lambda^\alpha (\eta/\sigma)}, \quad \lambda >0 \\
= &  \frac{1}{\lambda} \frac{\sigma}{\sigma\sqrt{\lambda} + \eta \lambda^\alpha } \frac{1}{\sqrt{\lambda}}
\end{align*}
coincides with \eqref{proofV2} and can be regarded as
\begin{align*}
\int_0^\infty e^{-\lambda s} \int_0^s \mathbf{P}_0(\bar{V}_t \leq s ) dt\, ds = \int_0^\infty e^{-\lambda s} \int_0^s \mathbf{P}_0(\bar{V}^{-1}_s \geq t ) dt\, ds
\end{align*}
where
\begin{align*}
\int_0^s \mathbf{P}_0(\bar{V}_s^{-1} \geq t) dt = \int_0^\infty \mathbf{P}_0(\bar{V}_s^{-1} \geq t) dt = \mathbf{E}_0[\bar{V}_s^{-1}], \quad s\geq 0.
\end{align*}
Formula \eqref{proofV2} leads to
\begin{align*}
& \int_0^\infty e^{-\lambda s} \int_0^s e^{-\xi t} \mathbf{P}_0(t + \gamma^{+, -1} \circ (\sigma/\eta)L_t \leq s) dt\, ds\\
= & \lim_{\xi \to 0} \frac{1}{\lambda (\xi +\lambda)} \frac{(\xi + \lambda)^\alpha}{(\xi + \lambda)^\alpha + \lambda^{1/2} (\sigma/\eta)}, \quad \lambda>0\\
= & \frac{1}{\lambda}\frac{\lambda^\alpha}{\lambda} \frac{\eta}{\lambda^\alpha + (\sigma/\eta)\sqrt{\lambda}}
\end{align*}
which coincides with \eqref{proofV1} and is the Laplace transform of
\begin{align*}
\int_0^s \mathbf{P}(\bar{\mathcal{V}}^{-1}_s \geq t) dt = \int_0^\infty \mathbf{P}(\bar{\mathcal{V}}^{-1}_s \geq t) dt = \mathbf{E}_0[\bar{\mathcal{V}}^{-1}_s], \quad s\geq 0.
\end{align*}

We are now ready for the next result.

\begin{theorem}
For $\lambda>0$,
\begin{align*}
\mathbf{E}_0\left[ \int_0^\infty e^{-\lambda t} \bar{V}^{-1}_t\, dt \right] = \mathbf{E}_0\left[ \int_0^\infty e^{-\lambda t} \int_0^t \mathbf{1}_{(0, \infty)}(\bar{X}_s) ds\, dt  \right]
\end{align*}
and
\begin{align*}
\mathbf{E}_0\left[ \int_0^\infty e^{-\lambda t} \bar{\mathcal{V}}^{-1}_t\, dt \right] = \mathbf{E}_0\left[ \int_0^\infty e^{-\lambda t} \int_0^t \mathbf{1}_{\{0\}}(\bar{X}_s) ds  \right].
\end{align*}
\end{theorem}
\begin{proof}
Let us take $f=\mathbf{1}_{(0,\infty)}$ in \eqref{uLapSolZero} with $c=0$. Then, 
\begin{align*}
\frac{1}{\lambda} \tilde{u}(\lambda, 0) = \mathbf{E}_0\left[ \int_0^\infty e^{-\lambda t} \int_0^t \mathbf{1}_{(0, \infty)}(\bar{X}_s) ds\, dt  \right] 
\end{align*}
which coincides with \eqref{proofV2}, that is with 
\begin{align*}
\int_0^\infty e^{-\lambda t} \mathbf{E}_0[\bar{V}^{-1}_t] dt.
\end{align*}
On the other hand, by taking 
\begin{align*}
f(x) = \left\lbrace
\begin{array}{ll}
0, & x \in (0, \infty)\\
1, & x=0 
\end{array}
\right .
\end{align*}
we obtain 
\begin{align*}
u(t, 0) 
= & \mathbf{E}_0[f(\bar{X}_t)]\\ 
= & \int_{[0, \infty)} f(y) \bar{p}(t,0,y) m(dy)\\ 
= & (\eta/\sigma) \int_{[0, \infty)} f(y) \bar{p}(t,0,y)\delta_0(dy)\\ 
= & (\eta/\sigma) \bar{p}(t, 0, 0).
\end{align*}
From \eqref{uLapSolZero} and the fact that
\begin{align*}
\frac{1}{\lambda} \tilde{u}(\lambda, 0) = \int_0^\infty e^{-\lambda t} \int_0^t u(s, 0) ds\, dt = \int_0^\infty e^{-\lambda t} \int_0^t \mathbf{E}_0[f(\bar{X}_s)] ds\, dt
\end{align*}
we write
\begin{align*}
\frac{1}{\lambda} \tilde{u}(\lambda, 0) = \int_0^\infty e^{-\lambda t} \mathbf{E}_0 \left[ \int_0^t \mathbf{1}_{\{0\}} (\bar{X}_s) ds  \right] dt
\end{align*}
which coincides with \eqref{proofV1} and therefore with
\begin{align*}
\int_0^\infty e^{-\lambda t} \mathbf{E}_0[\bar{\mathcal{V}}^{-1}_t] dt.
\end{align*}
\end{proof}

We study the occupation measure $\mu_{\bar{X}}$ defined as
\begin{align*}
\mu_{\bar{X}}(x, S) := \mathbf{E}_x\left[ \int_0^{\tau_{\bar{X}}(\Lambda)} \mathbf{1}_{S}(\bar{X}_s) ds \right], \quad x \in \Lambda \subseteq [0, \infty), \quad S \subset \Lambda, \quad \alpha \in (0,1]
\end{align*}
and for which
\begin{align*}
\mu_{\bar{X}}(x,\Lambda) = \mathbf{E}_x[\tau_{\bar{X}}(\Lambda)], \quad x \in \Lambda, \quad \alpha \in (0,1].
\end{align*}

\begin{lemma}
Let $\Lambda$ be an interval with $m(\Lambda)<\infty$.
\begin{itemize}
\item [i)] If $\Lambda \subset (0, \infty)$, then 
\begin{align*}
\forall\, \alpha \in (0,1], \quad \forall\, x \in \Lambda, \quad \mathbf{E}_x[\tau_{\bar{X}}(\Lambda) ]< \infty.
\end{align*} 
\item [ii)] Otherwise, 
\begin{align*}
\forall\, \alpha \in (0,1), \quad \forall\, x \in \Lambda, \quad \mathbf{E}_x[\tau_{\bar{X}}(\Lambda) ] = \infty
\end{align*}
and
\begin{align*}
\quad \forall\, x \in \Lambda, \quad \mathbf{E}_x[\tau_{X}(\Lambda) ] < \infty.
\end{align*}
\end{itemize} 
\end{lemma}
\begin{proof}
For the solution $u$ to the problem \eqref{P3-FBVP} we observe that
\begin{align*}
\lim_{\lambda \to 0} \tilde{u}(\lambda, x) = \int_0^\infty Q^\dagger_t f(x)\, dt + \lim_{\lambda \to 0} \tilde{u}(\lambda, 0), \quad x \in [0, \infty)
\end{align*}
where 
\begin{align*}
\int_0^\infty Q^\dagger_t f(x)\, dt = \mathbf{E}_x \left[\int_0^{\tau_0} f(\bar{X}_t) dt \right] = \mathbf{E}_x \left[\int_0^{\tau_0} f(X^\dagger_t) dt \right]
\end{align*}
and $\tilde{u}(\lambda, 0)$ is given in \eqref{uLapSolZero}. We consider $\Lambda=(0, \epsilon)$ with $\epsilon>0$ and $f=\mathbf{1}_\Lambda$. From \eqref{PotSemigDir} we obtain 
\begin{align*}
\mathbf{E}_x \left[\int_0^{\tau_0} \mathbf{1}_\Lambda(X^\dagger_t) dt \right] = x\int_0^\epsilon dy + 2x \int_0^\epsilon y\,dy = x \epsilon + x \epsilon^2, \quad x \in (0, \epsilon), \quad \forall \epsilon>0.
\end{align*}
From \eqref{potZeroEpsilon} we get
\begin{align*}
\lim_{\lambda \to 0} \tilde{u}(\lambda, 0) = (\sigma/c) \epsilon, \quad \forall \epsilon>0
\end{align*}
and therefore, we conclude that
\begin{align*}
\mathbf{E}_x[\tau_{\bar{X}}(\Lambda)] = x\epsilon + x \epsilon^2 + (\sigma/c) \epsilon < \infty \quad \forall\, x \in \Lambda \subset (0, \infty).
\end{align*}
For $\Lambda = (a,b) \subset (0, \infty)$ we use the fact that $\bar{X}$ behaves like $X^\dagger$ and the invariance of $X^\dagger$ with respect to translations. Thus, we write $\mathbf{E}_{x-a}[\tau_{\bar{X}}(\Lambda^\prime)]$ where $\Lambda^\prime = (0, b-a)$. This concludes the proof of $i)$.\\

From \eqref{potZeroEpsilon} we also get, for $\Lambda=[0, \epsilon)$, 
\begin{align*}
\lim_{\lambda \to 0} \tilde{u}(\lambda, 0) = (\sigma/c) \epsilon + (\eta/c) \lim_{\lambda \to 0} \frac{\lambda^\alpha}{\lambda}, \quad \alpha\in (0,1]
\end{align*}
which is finite only in case $\alpha=1$. That is the case $\bar{X}=X$ and 
\begin{align*}
\lim_{\lambda \to 0} \tilde{u}(\lambda, 0) = (\sigma/c) \, \int_{[0,\epsilon)} m(dx).
\end{align*}
This holds $\forall \, \Lambda$ bounded such that $\Lambda \ni \{0\}$ and proves that $ii)$ holds true.
\end{proof}

\begin{remark}
Consider $\Lambda = [0, \epsilon)$ with $\epsilon>0$. We observe that, for $\tau_{X^+}(\Lambda) = \inf \{t\,:\, X^+_t \notin \Lambda\}$ and $\tau_{X}$ as above, after simple calculation,  
\begin{align*}
\mathbf{E}_x \left[ \int_0^{\tau_{X^+}(\Lambda)} \mathbf{1}_{\Lambda}(X^+_t) dt \right] = \frac{\epsilon^2 - x^2}{2}
\end{align*}
and
\begin{align*}
\mathbf{E}_x \left[ \int_0^{\tau_{X}(\Lambda)} \mathbf{1}_{\Lambda}(X_t) dt \right] = \frac{\epsilon^2 - x^2}{2} + (\eta/\sigma) (\epsilon - x) .
\end{align*}
Thus, the (average) extra time the process $X$ spends on $\{0\}$ is given by
\begin{align*}
\mathbf{E}_x \left[ \int_0^{\tau_{X}(\Lambda)} \mathbf{1}_{\Lambda}(X_t) dt \right] - \mathbf{E}_x \left[ \int_0^{\tau_{X^+}(\Lambda)} \mathbf{1}_{\Lambda}(X^+_t) dt \right]= (\eta/\sigma) (\epsilon - x), \quad x \in \Lambda.
\end{align*}
We observe that
\begin{align*}
\mathbf{E}[e_0] = \eta/\sigma
\end{align*}
and, for the local time $\gamma^+$ accumulated up to time $\tau_\epsilon = \tau_{X^+}(\Lambda)$, from \cite[Theorem 7.7]{ChugWilliams},
\begin{align*}
\mathbf{E}_x [\gamma^+ \circ \tau_\epsilon] = \epsilon - x.
\end{align*}
We recall that $\gamma^+$ is the local time at zero of the process $X^+$ started at $X^+_0=x$. Thus, the extra time above is given by the holding time of $X$ and satisfies the Wald identity 
\begin{align*}
\mathbf{E}_x[Holding\ time] = \mathbf{E}[e_0]\, \mathbf{E}_x[\gamma^+ \circ \tau_\epsilon] = \mathbf{E}[e_0]\, \mathbf{E}_x \left[ \int_0^{\tau_\epsilon} d\gamma^+_t \right], \quad x \in \Lambda.
\end{align*}
\end{remark}

\begin{lemma}
Let $\Lambda$ be an interval with $m(\Lambda)<\infty$.
\begin{itemize}
\item[i)] If $\Lambda \subset (0, \infty)$, then
\begin{align*}
\forall\, \alpha \in (0,1], \quad \forall\, x \in \Lambda, \quad \mu_{\bar{X}}(x, S) < \infty.
\end{align*}
\item[ii)] Otherwise, 
\begin{align*}
\quad \forall\, x \in \Lambda, \quad  \mu_{X}(x, S) < \infty \quad \textrm{and} \quad   \forall\, \alpha \in (0,1), \quad \forall\, x \in \Lambda, \quad \mu_{\bar{X}}(x,S) \leq \infty.
\end{align*}
\end{itemize}
\end{lemma}
\begin{proof}
$\forall x \in \Lambda$ with $\Lambda$ a bounded subsect of $(0, \infty)$, $S \subset \Lambda$,  
\begin{align*}
\mu_{X}(x, S) 
= & \mathbf{E}_x\left[ \int_0^{\tau_X(\Lambda)} \mathbf{1}_S (X_s)ds \right] \\
\leq & \mathbf{E}_x\left[ \int_0^\infty \mathbf{1}_S (X_s)ds \right] \\
= & \lim_{\lambda \to 0} \mathbf{E}_x\left[ \int_0^\infty e^{-\lambda t} \mathbf{1}_S (X_s)ds \right]
\end{align*}
where (see formula \eqref{uLapSolZero})
\begin{align*}
\mathbf{E}_x\left[ \int_0^\infty e^{-\lambda t} \mathbf{1}_S (X_s)ds \right] 
= : & R_\lambda \mathbf{1}_S (x)\\  
\leq & \frac{\sigma \, m(S)}{c + \eta \lambda + \sigma \sqrt{\lambda}} \to (\sigma / c) m(S)< \infty \quad \textrm{as} \quad \lambda \to 0
\end{align*}
where $m(S)\leq m(\Lambda)$. With the same arguments we obtain
\begin{align*}
\mu_{\bar{X}}(x, S) 
= & \mathbf{E}_x\left[ \int_0^{\tau_{\bar{X}}(\Lambda)} \mathbf{1}_S (\bar{X}_s)ds \right]\\ 
\leq & \frac{\sigma m(S)}{c + \eta \lambda^\alpha + \sigma \sqrt{\lambda}} \to (\sigma / c) m(S) < \infty \quad \textrm{as} \quad \lambda \to 0.
\end{align*}
Both limits above are obtained uniformly in $x \in [0, \infty)$ and this concludes the proof of $i)$.\\

For $\Lambda \ni \{0\}$, for example $\Lambda=[0, \epsilon)$ with $\epsilon>0$, we have that
\begin{align*}
\mu_{\bar{X}}(x, \Lambda) \leq \mathbf{E}_x\left[ \int_0^\infty \mathbf{1}_\Lambda (\bar{X}_s)ds \right] = \lim_{\lambda \to 0} \mathbf{E}_x\left[ \int_0^\infty e^{-\lambda t} \mathbf{1}_\Lambda (\bar{X}_s)ds \right]
\end{align*}
where (see formula \eqref{probIntZero})
\begin{align*}
\mathbf{E}_x\left[ \int_0^\infty e^{-\lambda t} \mathbf{1}_\Lambda (\bar{X}_s)ds \right] 
= \frac{\sigma}{c + \lambda^\alpha \eta + \sigma \sqrt{\lambda}} \frac{1 - e^{-\epsilon \sqrt{\lambda}}}{\sqrt{\lambda}} + \frac{\lambda^\alpha}{\lambda} \frac{\eta}{c + \eta \lambda^\alpha + \sigma \sqrt{\lambda}}
\end{align*}
and
\begin{align*}
\lim_{\lambda \to 0} \mathbf{E}_x\left[ \int_0^\infty e^{-\lambda t} \mathbf{1}_\Lambda (\bar{X}_s)ds \right] = \lim_{\lambda \to 0} \frac{\lambda^\alpha}{\lambda} \frac{\eta}{c + \eta \lambda^\alpha + \sigma \sqrt{\lambda}} = \infty
\end{align*}
only in case $\alpha<1$ concluding the proof.
\end{proof}

\begin{remark}
We notice that
\begin{align*}
\forall\, S \subset \Lambda, \quad \mathbf{E}_x \left[ \int_0^{\tau_{\bar{X}}(\Lambda)} \mathbf{1}_{S}(\bar{X}_s) ds \right] \leq \mathbf{E}_x \left[ \int_0^{\tau_{\bar{X}}(\Lambda)} \mathbf{1}_{\Lambda}(\bar{X}_s) ds \right].
\end{align*}
The processes $X$ and $\bar{X}$ move along the path of $X^+$ according with their clocks $V$ and $\bar{V}$. Assume $c=0$. By definition $\bar{V}^{-1}_t \leq t$ and $V^{-1}_t \leq t$ a.s. and $\forall \, x \in \Lambda \subset (0, \infty)$,  $\bar{V}^{-1}_t = t$ and $V^{-1}_t=t$ for $t< \tau_0$, thus  
\begin{align*}
\tau_{\bar{X}}(\Lambda) = \tau_X(\Lambda) = \inf\{t\,: X^+_t \notin \Lambda \}.
\end{align*}  
Consider $\Lambda=(a,b) \ni x=X^+_0$ with $b > a > 0$ and write 
\begin{align*}
\tau_{X}(\Lambda) = \inf\{t\,: X^+_t < a \} \wedge \inf\{t\,: X^+_t >b \}.
\end{align*}
Since 
\begin{align*}
\inf\{t\,: X^+_t < a \}=\inf\{t\,: X^\dagger_t = a \}=: \tau_a
\end{align*}
and 
\begin{align*}
\inf\{t\,: X^+_t >b \}=\inf\{t\,: X^\dagger_t  = b  \}=: \tau_b,
\end{align*}
then we get
\begin{align*}
\mu_{\bar{X}}(x, \Lambda) 
= & \mathbf{E}_x \left[ \int_0^{\tau_a \wedge \tau_b} \mathbf{1}_{(a,b)} (X^\dagger_t) dt \right]\\ 
= & \mathbf{E}_x [\tau_a \wedge \tau_b] \\
= & -\frac{(x-a)^2}{2} + \frac{(b-a)(x-a)}{2}.
\end{align*}
This provides an alternative proof of the first statement of the previous Lemma.
\end{remark}

\subsection{On the holding times of $\bar{X}$}

We say that $\{e_i\}_{i}$ is a sequence of exponential holding times for $X$ meaning that, for $X_0=0$, 
\begin{align*}
\mathbf{P}_0(e_i >t | X_{e_i} >0) = e^{-(\sigma/\eta) t}.
\end{align*}
Despite the fact that $X$ is not strong Markov, the exponential law is directly related with the fact that $X$ is Markov. Let $\{\bar{e}_i\}_i$ be the sequence of holding time for $\bar{X}$. The process $\bar{X}$ is defined via time-change along the path of $X^+$. We have instantaneous reflection at zero for $X^+$ and a delayed reflection for $\bar{X}$. Indeed, the non-local boundary condition introduces a delaying effect due to the process $L$.

A sequence of holding times (at the boundary point) for $\bar{X}$ is a sequence of random variables, say $\{\bar{e}_i\}_i$, with distribution given by
\begin{align*}
\forall\, i, \quad \mathbf{P}_0(\bar{e}_i > t | \bar{X}_{\bar{e}_i}>0), \quad t>0.
\end{align*}
Thus the process started from $\bar{X}_0=0$ leaves the boundary point after an holding time $\bar{e}_i$.  We have that $\bar{X}$ moves along the path of $X^+$ which exhibits instantaneous reflection. The time change given by $\bar{V}_t = t + H \circ (\eta/\sigma) \gamma^+_t$ is the clock with the extra time $H \circ (\eta/\sigma) \gamma^+_t$ depending on the time the process $X^+$ spend at zero up to time $t$. In particular $\bar{V}$ plays the same role of $V$ for which the extra time $(\eta/\sigma)\gamma^+_t$ leads to the holding time $e_i$, $i \in \mathbb{N}$ for $X$.

\begin{theorem}
\label{thm:holdingTime}
The sequence $\{\bar{e}_i\}_{i}$ of holding times for $\bar{X}$ is given by i.i.d. Mittag-Leffler random variables. 
\end{theorem}
\begin{proof}
By definition, $\bar{e}_0$  is such that $\bar{X}_t= 0$ for $0\leq t < \bar{e}_0$ if $\bar{X}_0=0$. Focus on the case $\alpha=1$. Assume $X_0=0$, then $X_t = 0$ for $0\leq t < e_0$ by definition of holding time $e_0$ for $X$. The process $V_t=t+(\eta/\sigma) \gamma^+_t$ has continuous paths as well as the inverse process $V^{-1}_t$. As $\eta=0$ we get $V_t=t=V^{-1}_t$ and the process $X=X^+$ has instantaneous reflection. For $\eta>0$, the extra time $(\eta/\sigma) \gamma^+_t$ of $V_t$ introduces the holding time $e_0$ via time change with $V^{-1}$. In particular, 
\begin{align*}
\forall\, t\geq 0, \;  V_t - t \geq 0 \quad \textrm{and} \quad V_t - t > 0 \quad \textrm{if} \quad 0 < t < e_0 
\end{align*}
and 
\begin{align*}
\mathbf{P}_0(0 < V_t - t < e_0) = \mathbf{P}_0((\eta/\sigma) \gamma^+_t < e_0).
\end{align*}
With $\tau_0=0$, let $\tau_i = \inf\{t > \tau_{i-1}\,:\, X_t = 0 \}$, $i \in \mathbb{N}$ be the sequence of return times at zero for $X$ and define
\begin{align*}
N_t = \max\{i\,:\, \tau_i < t \}.
\end{align*}
The time $t$ is a clock for $X^+$ as well as the time $V_t$ is a clock for $\bar{X}$. The difference $V_t - t$ is given by the amount of time the process $\bar{X}$ has been stopped at zero up to time $t$. That is,
\begin{align*}
V_t - t = \sum_{i=0}^{N_t} e_i = \sum_{i\geq 0} e_i \, \mathbf{1}_{(i \leq N_t)} = \sum_{i\geq 0} e_i \, \mathbf{1}_{(\tau_i < t)}
\end{align*}
and the first difference $e_0$ between the two clocks is obtained at the first instantaneous visit of $X^+$ at zero. Moreover, we can write
\begin{align*}
\bar{V}_t - t = \sum_{i\geq 0} \bar{e}_i \, \mathbf{1}_{(\bar{\tau}_i < t)}
\end{align*}
which gives $\bar{e}_0$ according with the previous case. Notice that $\bar{X}_0=0$ and $\bar{\tau}_0 = 0$, $X$ and $\bar{X}$ move along the path of $X^+$ and $\bar{X} = X$ for $\alpha=1$. Now we observe that
\begin{align*}
V_t - t = (\eta/\sigma) \gamma^+_t = e_0, \quad 0 \leq t < \tau_1 
\end{align*}
and therefore
\begin{align*}
\bar{V}_t - t = H \circ (\eta/\sigma) \gamma^+_t = \bar{e}_0, \quad 0 \leq t < \tau_1  \quad \Rightarrow \quad \bar{V}_t - t = H \circ e_0, \quad 0 \leq t < \tau_1,
\end{align*}
that is, $\bar{e}_0 = H \circ e_0$. Observe that $\tau_1$ is the excursion time on $(0, \infty)$ for a Brownian motion, thus both processes $V_t - t$ and $\bar{V}_t -t$ run over $[0, \tau_1)$ respectively taking the values $e_0$ and $\bar{e}_0$. Since $e_0$ is and exponential r.v. we conclude that
\begin{align*}
\mathbf{E}_0[e^{-\lambda H\circ e_0}] = \int_0^\infty (\sigma/\eta) e^{-(\sigma/\eta) s} e^{- \lambda^\alpha s} ds = \frac{(\sigma/\eta)}{(\sigma/\eta) + \lambda^\alpha}
\end{align*} 
and $H\circ e_0$ has a Mittag-Leffler distribution. Indeed, 
\begin{align*}
\int_0^\infty e^{-\lambda t} E_\alpha(-(\sigma/\eta) t^\alpha) dt = \frac{\lambda^\alpha}{\lambda} \frac{1}{(\sigma/\eta) + \lambda^\alpha}, \quad \lambda >0
\end{align*}
and
\begin{align*}
\mathbf{P}_0(H\circ e_0 > t) = E_\alpha(-(\sigma/\eta) t^\alpha).
\end{align*} 
Since $e_i \sim e_0$ are i.i.d. random variables, then $H\circ e_i$ are i.i.d. random variables. Indeed, $H$ is independent from $X$. Moreover, $H$ is a Markov process for which
\begin{align*}
H \circ e_0 \perp H \circ e_{1} = H\circ (e_0+e_1) - H \circ e_0 
\end{align*}
as well as $H \circ e_i \perp H \circ e_k$ $\forall\, i\neq k$. Thus, we obtain the claim.
\end{proof}

\begin{corollary}
The process $\bar{X}$ spends an infinite average amount of time at $\{0\}$ with each visit.
\end{corollary}
\begin{proof}
It follows from the fact that $E_\alpha \notin L^1(0, \infty)$ and
\begin{align*}
\forall\, i \in \mathbb{N}_0, \quad \mathbf{E}[\bar{e}_i] = \int_0^\infty E_\alpha (-(\sigma/\eta)t^\alpha) dt.
\end{align*}
\end{proof}

\subsection{On the lifetime of $\bar{X}$}

We provide some connection between the results obtained in \cite{Dov22} and the results obtained in the present paper. Let us introduce 
\begin{align*}
\bar{H}_t = \frac{\sigma}{\eta} t + H^\beta_t, \; t\geq 0 \quad \textrm{with inverse} \quad \bar{L}_t, \; t \geq 0
\end{align*}
where $H^\beta = \{H^\beta_t\}_{t\geq 0}$ is a $\beta$-stable subordinator with inverse $L^\beta = \{L^\beta_t\}_{t\geq 0}$. We still denote by $H$ and $L$ the $\alpha$-stable subordinator and its inverse. Let $\chi$ be an exponential random variable such that $\mathbf{P}(\chi >x)= e^{-cx}$, $x>0$. We also introduce $\chi(\beta)$ which is a Mittag-Leffler random variable with $\mathbf{P}(\chi(\beta) > x) = E_\beta(-c x^\beta)$, $c>0$. It is well-know that $\chi(1)=\chi$.

In \cite{Dov22} the FBVP \eqref{FBVP-realLine} has been investigated for $\alpha=\beta/2 \in (0, 1/2]$ and an alternative probabilistic representation has been considered. Such a representation has been related to the fractional telegraph process. Let us consider the lifetime 
\begin{align*}
\zeta \stackrel{law}{=} \inf\{t\geq 0\,:\, \gamma^+_t < \chi(\beta) \} \quad \textrm{with } \beta/2 = \alpha\leq 1/2.
\end{align*}
For $f=\mathbf{1}$ and $\sigma=0$, from \cite{Dov22} we known that  
\begin{align*}
u(t,x) = \mathbf{P}_x(\zeta >t) = \mathbf{E}_x \left[ \exp \left( - (c/\eta) L^{\beta} \circ \gamma^+_t \right) \right] = \mathbf{E}_x \left[ E_{\beta} \big( - (c/\eta) (\gamma^+_t)^{\beta} \big) \right]
\end{align*}
which is the case of the elastic sticky fractional condition. Moreover, the special case $\alpha=1/2$ gives
\begin{align*}
u(t,x) = \mathbf{E}_x\left[\exp \left( - \frac{c}{\eta + \sigma} \gamma^+_t \right) \right]
\end{align*}
which is the case of the elastic condition. Thus, for $\alpha=1/2$ and $\sigma=0$, we have that
\begin{align*}
\zeta \stackrel{law}{=} \inf\{t\geq 0\,:\, \gamma^+_t < \chi \}
\end{align*}
coincides with the lifetime of an elastic Brownian motion. This relates the Neumann boundary condition to the fractional dynamic boundary condition with $\alpha=1/2$. 

\begin{remark}
Focus on the case $\sigma=0$. We have
\begin{align*}
u(t,x) 
= & \mathbf{E}_x \left[ \exp \left( -(c/\eta)\, L^{\beta} \circ \gamma^+_t \right) \right]\\ 
= & \mathbf{E}_x \left[ E_{\beta} \big( - (c/\eta) (\gamma^+_t)^{\beta} \big) \right]\\ 
= & \int_0^\infty E_{\beta}(-(c/\eta) s^{\beta}) \mathbf{P}_x(\gamma^+_t \in ds)
\end{align*}
for which
\begin{align*}
\frac{\partial u}{\partial t}(t,x) 
= & \int_0^\infty E_{\beta}(- (c/\eta) z^{\beta}) \frac{\partial^2}{\partial z^2} \frac{e^{-(x+z)^2/4t}}{\sqrt{4\pi t}} dz \\
= & -E_\beta(-(c/\eta) z^\beta) \frac{(x+z)}{2t} \frac{e^{-(x+z)^2/4t}}{\sqrt{4\pi t}} \bigg|_0^\infty \\
& + \int_0^\infty \frac{\partial}{\partial z} E_\beta(- (c/\eta) z^\beta) \frac{(x+z)}{2t} \frac{e^{-(x+z)^2/4t}}{\sqrt{4\pi t}} dz\\
= & \frac{x}{2t} \frac{e^{-x^2/4t}}{\sqrt{4\pi t}} +  \int_0^\infty \frac{\partial}{\partial z} E_\beta(- (c/\eta) z^\beta) \frac{(x+z)}{2t} \frac{e^{-(x+z)^2/4t}}{\sqrt{4\pi t}} dz.
\end{align*}
Thus,
\begin{align*}
\bigg| \frac{\partial u}{\partial t}(t,0) \bigg|
= & \bigg| \int_0^\infty \frac{\partial}{\partial z} E_\beta(-(c/\eta) z^\beta) \frac{z}{2t} \frac{e^{-z^2/4t}}{\sqrt{4\pi t}} dz \bigg| \\
\leq & \int_0^\infty \bigg|\frac{\partial}{\partial z} E_\beta(- (c/\eta) z^\beta) \bigg| \frac{z}{2t} \frac{e^{-z^2/4t}}{\sqrt{4\pi t}} dz.
\end{align*}
According with \cite[equation (17)]{Kra03}, we have that, for $\mathtt{C}>0$,
\begin{align*}
\bigg| \frac{\partial u}{\partial t}(t,0) \bigg|
\leq & (c/\eta) \int_0^\infty \mathtt{C}\, z^{\beta -1} \frac{z}{2t} \frac{e^{-z^2/4t}}{\sqrt{4\pi t}} dz\\
= & [z=\sqrt{y}] = (c/\eta)\, \mathtt{C} \frac{1}{2} \frac{1}{2t} \frac{1}{\sqrt{4\pi t}} \int_0^\infty y^{\beta/2 + 1/2 -1} e^{-y/4t} dy \\
= & (c/\eta)\,\mathtt{C} \frac{1}{2} \frac{1}{2t} \frac{1}{\sqrt{4\pi t}} (4t)^{\beta/2 + 1/2} \Gamma(\beta/2 + 1/2)\\
= & (c/\eta)\, \mathtt{C} \frac{1}{2t} \frac{1}{\sqrt{4\pi t}} (4t)^{\beta/2 + 1/2} 2^{1-\beta} \frac{\Gamma(\beta)}{\Gamma(\beta/2)} \sqrt{\pi}\\
= & (c/\eta)\, \mathtt{C}\, 2^{- \beta} \frac{\Gamma(\beta)}{\Gamma(\beta/2)} t^{\beta/2 -1}
\end{align*}
with $\beta/2 \in (0, 1/2]$. This agrees with the formula \eqref{estMmodulus} which writes
\begin{align}
\bigg| \frac{\partial u}{\partial t}(t, 0) \bigg| \leq \mathcal{C}\, \|f\|_\infty\, \frac{\sqrt{(c/\eta)}}{2}  t^{\alpha/2-1}, \quad t>0.
\end{align}

\end{remark}

Let us denote by $\bar{\zeta}$ the lifetime of $\bar{X}$. For the process $\bar{V}_t = t + H \circ (\eta / \sigma) \gamma^+_t$ defined as above we provide the following result.
\begin{theorem}
For $\alpha= \beta/2 \in (0, 1/2)$, 
\begin{align}
\mathbf{P}_x(\bar{\zeta} >t) = \mathbf{P}_x(\chi > \bar{L} \circ \gamma^+_t) = \mathbf{P}_x(\chi > \gamma^+ \circ \bar{V}^{-1}_t), \quad t>0, \quad x \geq 0
\end{align}
and
\begin{align}
\bar{\zeta} \stackrel{law}{=} \inf\{t \geq 0\,:\, \gamma^+_t < \bar{H}_\chi \}.
\end{align}
For $\alpha \in (0,1)$
\begin{align}
\label{lifeBarAllalpha}
\mathbf{P}_x(\bar{\zeta} >t) = \mathbf{P}_0(H^{1/2}_x + H^{1/2}_\chi + (\eta/\sigma)^{1/\alpha} H^\alpha_\chi > t), \quad t>0, \quad x \geq 0
\end{align}
where $H^{1/2}_x$, $H^{1/2}_\chi$, $H^\alpha_\chi$, $\bar{X}$ are independent.
\end{theorem}
\begin{proof}
Focus on 
\begin{align*}
\mathbf{E}_x \left[\exp\left( -c \gamma^+ \circ \bar{V}^{-1}_t \right) \right] = \int_0^\infty e^{-cs} \mathbf{P}_x(\gamma^+ \circ \bar{V}^{-1}_t \in ds).
\end{align*}
Let us consider
\begin{align*}
\mathbf{P}_x(\gamma^+ \circ \bar{V}^{-1}_t > s) = \mathbf{P}_x(\bar{V}^{-1}_t > \gamma^{+,-1}_s) = \mathbf{P}_x(t > \bar{V} \circ \gamma^{+,-1}_s).
\end{align*}
Notice that $\gamma^+_t$ is non decreasing and therefore the process $\gamma^{+,-1}$ may have jumps. Thus, $\gamma^{+,-1} \circ \gamma^+_t$ is a little tricky whereas $\gamma^+ \circ \gamma^{+,-1}_t = t$ almost surely. We can therefore write
\begin{align*}
\mathbf{P}_x(t > \bar{V} \circ \gamma^{+,-1}_s) 
= & \mathbf{P}_x(t> \gamma^{+,-1}_s + (\eta/\sigma)^{1/\alpha} H_s)\\ 
= & \mathbf{P}_x(t - (\eta/\sigma)^{1/\alpha} H_s > \gamma^{+,-1}_s)\\
\displaystyle = & \int_0^t \mathbf{P}_x( t - z > \gamma^{+,-1}_s) \mathbf{P}_0((\eta/\sigma)^{1/\alpha} H_s \in dz)
\end{align*}
which is a convolution with Laplace transform
\begin{align}
&\int_0^\infty e^{-\lambda t} \mathbf{P}_x(\gamma^+ \circ \bar{V}^{-1}_t > s) dt \\
= & \left( \int_0^\infty e^{-\lambda t} \mathbf{P}_x(t > \gamma^{+,-1}_s) dt \right) \left(\int_0^\infty e^{-\lambda t} \mathbf{P}_0((\eta/\sigma)^{1/\alpha} H_s \in dt) \right)\notag \\
= & \int_0^\infty e^{-\lambda t} \mathbf{P}_x(\gamma^+_t > s) dt  \, e^{-s \lambda^\alpha \eta /\sigma}\notag \\
= & \int_s^\infty \frac{\sqrt{\lambda}}{\lambda} e^{-w \sqrt{\lambda}} dw\,  e^{-s \lambda^\alpha \eta /\sigma}\, e^{-x \sqrt{\lambda}}\notag \\
= & \frac{1}{\lambda} e^{-s \sqrt{\lambda} - s \lambda^\alpha \eta/\sigma} \, e^{-x \sqrt{\lambda}}, \quad \lambda>0.
\label{VbarInv}
\end{align}
This implies that
\begin{align*}
\int_0^\infty e^{-\lambda t} \mathbf{P}_x(\gamma^+ \circ \bar{V}^{-1}_t \in ds)\, dt = \frac{\sqrt{\lambda} + \lambda^\alpha \eta / \sigma}{\lambda} e^{-s \sqrt{\lambda} - s \lambda^\alpha \eta /\sigma}\, e^{-x\sqrt{\lambda}} ds, \quad \lambda>0.
\end{align*}
Since, for the independent subordinators $H^{1/2}_x$, $H^{1/2}_s$, $H^\alpha_s$,  
\begin{align*}
& \frac{1}{\lambda} \mathbf{E}_0 e^{-\lambda H^{1/2}_x -\lambda H^{1/2}_s - \lambda (\eta/\sigma)^{1/\alpha} H^\alpha_s} \\
= & \int_0^\infty e^{-\lambda t} \mathbf{P}_0(H^{1/2}_x + H^{1/2}_s + (\eta/\sigma)^{1/\alpha} H^\alpha_s < t) dt
\end{align*}
from \eqref{VbarInv}, we get that
\begin{align*}
\mathbf{P}_x(\gamma^+ \circ \bar{V}^{-1}_t > s) = \mathbf{P}_0(H^{1/2}_x + H^{1/2}_s + (\eta / \sigma)^{1/\alpha} H^\alpha_s < t)
\end{align*}
and
\begin{align*}
\mathbf{P}_x(\gamma^+ \circ \bar{V}^{-1}_t < \chi ) 
= & \mathbf{P}_0(H^{1/2}_x + H^{1/2}_\chi + (\eta/\sigma)^{1/\alpha} H^\alpha_\chi > t), \quad t>0,\; x \geq 0.
\end{align*}
This concludes the proof.
\end{proof}

\begin{remark}
From \eqref{lifeBarAllalpha} we see that
\begin{align*}
\mathbf{E}_x[\bar{\zeta}] = \mathbf{E}_0[H^{1/2}_x] + \mathbf{E}\Big[ \mathbf{E}_0[H^{1/2}_\chi | \chi ] + (\eta/\sigma)^{1/\alpha} \mathbf{E}_0[ H^\alpha_\chi | \chi] \Big]
\end{align*}
is infinite for $\alpha \in (0,1]$ also in case $c>0$, that is the case of elastic kill. \end{remark}

{\bf Acknowledgements}
 The author thanks his institution for the support under the Ateneo Grant 2020 and MUR for the support under PRIN 2022 - 2022XZSAFN: Anomalous Phenomena on Regular and Irregular Domains: Approximating Complexity for the Applied Sciences - CUP B53D23009540006.\\
 Web Site: \url{https://www.sbai.uniroma1.it/~mirko.dovidio/prinSite/index.html}.

\end{document}